\documentclass[numbers]{imaiai}

\usepackage{graphics} 
\usepackage{algorithmic} 

\usepackage{amstext,amsthm,amsfonts,latexsym,amssymb,graphicx,color,amsmath}
\usepackage{flushend}
\usepackage[usenames,dvipsnames]{xcolor}
\usepackage[small]{caption}
\usepackage{flushend}
\usepackage{color}
\usepackage{comment}
\usepackage{url}\usepackage{subfigure}
\DeclareMathOperator{\sgn}{sgn}
\def\qed{\hbox{${\vcenter{\vbox{		 
   \hrule height 0.4pt\hbox{\vrule width 0.4pt height 6pt
   \kern5pt\vrule width 0.4pt}\hrule height 0.4pt}}}$}}

\def\THEN{\texttt{THEN}}
\def\IF{\texttt{IF}}
\def\ELSE{\texttt{ELSE}}
\def\REPEAT{\texttt{REPEAT}}
\def\END{\texttt{END}}
\def\UNTIL{\texttt{UNTIL}}
\def\AND{\texttt{AND}}
\def\TRUE{\texttt{TRUE}}
\def\WHILE{\texttt{WHILE}}

\def\cH{\mathcal H}

\def\bP{\mathbb P}

\def\eps{\varepsilon}

\def\eps{\varepsilon}

\def\ER{Erd\H{o}s-R\'enyi }

\def\diag{\mathop{\rm diag}}

\newcommand{\fdp}[1]{\textcolor{black}{#1}}
\newcommand{\stef}[1]{\textcolor{black}{#1}}
\newcommand{\fdpe}[1]{\textcolor{black}{#1}}
\newcommand{\steff}[1]{\textcolor{black}{#1}}
\newcommand{\rev}[1]{\textcolor{black}{#1}}
\newcommand{\fdpr}[1]{\textcolor{black}{#1}}

\begin{document}

\title{Optimal curing policy for epidemic spreading over a community network with heterogeneous population}

\shorttitle{Optimal curing policy for epidemic spreading over a community network} 
\shortauthorlist{S. Ottaviano, F. De Pellegrini, S. Bonaccorsi, P. Van Mieghem} 

\author{{
\sc Stefania Ottaviano}
\\[2pt]
Fondazione Bruno Kessler, Via alla Cascata 56/d, 38123 Povo (Trento), Italy\\
Mathematics Department, University of Trento,  Via Sommarive 14, 38123 Povo (Trento), Italy\\[2pt]
{\sc Francesco De Pellegrini}\\[2pt]
Fondazione Bruno Kessler, Via alla Cascata 56/d, 38123 Povo (Trento), Italy\\[2pt]
 {\sc Stefano Bonaccorsi}\\[2pt]
Mathematics Department, University of Trento,  Via Sommarive 14, 38123 Povo (Trento), Italy\\[2pt]
{\sc and}\\[6pt]
{\sc Piet Van Mieghem} \\[2pt]
EEMCS, Delft University of Technology, Mekelweg 4 2628 CD Delft, Netherlands}

\maketitle

\begin{abstract}
{The design of an efficient curing policy, able to stem an epidemic process at an affordable cost, has to account for the structure of the population contact network supporting the 
    contagious process. Thus, we tackle the problem of   
allocating recovery resources among the population, at the lowest cost possible to prevent the epidemic from persisting indefinitely in the network. Specifically, we analyze a susceptible--infected--susceptible
    epidemic process spreading
    over a weighted graph, by means of a first-order mean-field approximation. First, we describe the influence of the contact network on the dynamics of the epidemics among a heterogeneous population, that is possibly divided into communities.
		For the case of a community network, our investigation relies on the graph-theoretical notion 
of equitable partition;  
	we show that the epidemic threshold, a key measure of the network robustness against
epidemic spreading, can be determined using a lower-dimensional 
dynamical
 system. 
Exploiting the computation of the epidemic threshold, we determine a cost-optimal curing policy by solving a convex minimization problem, which possesses a reduced dimension in the case of a community network. 
Lastly, we consider a two-level optimal curing problem, for which an algorithm is designed 
with a polynomial time complexity 
 in the network size.}
{heterogeneous SIS model, community network, graph spectra, equitable partitions, convex optimization.}
\\
2000 Math Subject Classification: 82C20, 34A34, 34D23, 90C22, 90C25  
\end{abstract}

\section{Introduction}\label{intro}

The \fdp{diffusion} and persistence of infectious diseases depend on complex
 interactions between individual \fdp{units (namely people, cities, countries, etc.),
   the characteristics of a disease and, possibly, on the applied control policies.
   The last ones are aimed at arresting disease transmission or render the infection
   prevalence as low as possible. } 
	

 Epidemic models have been used to describe a wide range of other phenomena \fdp{as well}, like social
 behaviors, diffusion of information, computer viruses, etc., indeed, although the basic mechanisms
 of these phenomena can be different, often their dynamical behavior can be described by the same
 type of equations \cite{PietSurvey}. \fdp{One of the main objectives, in all these domains, is
  to gain insight into how the spreading process transmits and to identify the most effective strategies
  in order to prevent and control them.}

%
In
 controlling the diffusion processes, the structure of the contact network plays a
 crucial role. \fdp{In particular, several contact
   networks appear organized into communities}.
	%
In this framework, a uniform control strategy not always represents the most effective way to
 reduce the infection rate, the number of affected individuals or the time of extinction~\cite{wan2008designing,prakash2013fractional,borgs2010distribute}. \steff{Furthermore, curing costs may vary from node to node.} In the case of community networks, curing costs \fdp{may vary depending on the features of the specific community
   where curing controls are applied.} 

\fdp{Thus, by taking into} account the topology of a community network, in this work we want to determine a cost-optimal distribution of resources, that is able to prevent the disease from persisting
 indefinitely in the population. \steff{The non-uniform distribution of resources aims to control, in a cost-optimal way, the level of the nodes local curing rates. Increasing the curing rates of, e.g., some selected communities, is reflected into speeding up their detection capabilities and treatments (or, into installing better virus scan software, in the
case of computer viruses) \cite{wan2008designing}}.

~\steff{The problem of designing strategies
  to \fdp{stop spreading processes} in networks has been largely tackled.} \fdp{Though, in this context,} \steff{to the best of our knowledge}, \fdp{very few
  works have described how to exploit the community structures in order to} \steff{formulate an optimization problem for resources allocation, with lower complexity}.
Based on the theory of contact processes,~\steff{Borgs et al. \cite{borgs2010distribute} characterize the optimal
  distribution of a fixed amount of antidote in a given network.
Gourdin et al. \cite{gourdin2011optimization} and Sahnen et al. \cite{sahneh2012optimal} take advantage of the
$N$-intertwined approximation \cite{VanMieghem2009} to analyze and control the spread of a SIS epidemic model.
The same mean-field approach is adopted by Preciado et al in \cite{Pappas}, where \fdp{the authors propose}
a semidefinite programming (SDP) approach for optimal network immunization.} \fdp{Cost-optimal vaccine allocation
in arbitrary undirected networks are obtained via the minimization of a vaccination cost function which depends
on infection rates.} In \cite{preciado2014optimal}, the same authors specialize some specific instances of
optimal network protection problem, via Geometric Programming techniques, to weighted, directed, strongly
(and not necessarily strongly) connected networks to compute the cost and speed optimal allocation of vaccines
and/or antidotes. In Sec. \ref{opt}, we analyze more in detail the differences between their and our approach.
 \steff{Enyioha et al. \cite{enyioha2015distributed} propose a distributed solution to the vaccine and
  antidote allocation problem to contain an epidemic outbreak in the absence of a central social planner.
  Each node locally computes its optimum investment in vaccine and antidotes needed to globally contain
  the spread of an outbreak, via local exchange of information with neighbors. %
  Drakopoulos et. al~\cite{drakopoulos2015network,drakopoulos2015lower} consider the propagation
  of an epidemic process \fdp{over a} network and study the problem of dynamically allocating a
  fixed curing budget to the nodes of the graph. The objective is to minimize the expected
  extinction time of the epidemics. In the case of bounded degree graphs, they provide a lower bound
  on the expected time to extinction under any such dynamic allocation policy}.
	
\steff{\subsection{Outline and main results}}
\fdp{Compared to previous works on optimal curing policy, we are interested, particularly, in leveraging
  the subdivision of the population into communities. The motivation comes from the fact that
  community structures are a relevant non-trivial topological feature of complex networks.
  Community structures have been identified as a typical feature of social networks, tightly connected groups of
nodes in the World Wide Web usually correspond to
pages on common topics, whereas
  in the biology framework, e.g., in cellular and genetic networks, communities may
  relate to functional modules~\cite{boccaletti2006}.} 
\fdp{Consequently, in many practical situations, it appears reasonable to consider curing policies which apply {\em per 
community} (i.e., per hospital, school, village, or city, etc,...), rather than policies which apply 
  per individual unit.}

\steff{In particular, we consider a continuous-time susceptible--infected--susceptible (SIS) epidemic process, where an individual can be repeatedly infected, 
recover and yet be infected again.
An input weighted graph captures the interaction between individuals and communities, where
  the heterogeneity, and possible asymmetries in the contagiousness, are caught by edge-dependent weights.} 

\steff{Our investigation, on a population divided into communities, has been based on the graph-theoretical notion of \emph{equi\-table partition}~\cite{Schwenk,godsil,EvolDelio}. 
A network with an equitable partition of its node set posses some interesting symmetry properties; we will use the word ``symmetry'' to refer to a certain structural regularity of the graph connectivity \cite{EvolDelio}.
We take advantage of the notion of equitable partitions for providing curing policies, diversified for communities, capable to lead the system to extinction, at the minimum cost.} In this context, our main goal is that we are able to formulate a convex cost minimization problem with a reduced dimension, with respect to the general case, where curing policies are providing for each node.

\steff{Spatial inhomogeneity has been incorporated in other mo\-dels to study the epidemic control \cite{watts2005multiscale,may1984spatial,riley2003transmission}, however not much effort has been made to explore inhomogeneous control strategies within this kind of models \cite{wan2008designing}}. \steff{The problem of an inhomogeneous allocation of limited resources for a multi-group model, 
has been studied, instead, in
  \cite{wan2008designing} by Wan et al. The aim of the authors is to maximize the speed at which the virus is eliminated. Thus, considering a discrete-time epidemic model, they tackle the problem to minimize the dominant eigenvalue of a system matrix, subject to limiting constraints on
  some system parameters to be controlled.}
\steff{Compared to their formulation, we want
to allocate resources to the communities, sufficient to lead the system towards the epidemic extinction, with the aim of minimizing a certain
  cost function.
	Moreover, 
	in the cited paper, individuals transmit the disease
  through homogeneous mixing within their own group, as well as 
interactions with individuals in other groups,
like in the usual metapopulation models. Such model is only a specific case of our network model, in fact, by the notion of equitable partition, we go beyond the full mesh assumption, within the communities, as well as outside, thus providing results for wider possible scenario (see Section \ref{sec:comm}).}\\\\


%
%
%



The paper is organized as follows.
In Sec. \ref{enm}, first, we review some background concepts for epidemic processes on networks in the homogeneous
  setting and the related mean-field approximation adopted in the paper.
	Then, we provide the adaptation of the model to the heterogeneous setting and we report on the analysis of the global dynamics of the epidemic process.
	This allows to recognize the stability modulus
of a matrix, encoding for the network structure
and for the parameters of the model, as the critical value separating an absorbing phase, from an endemic phase.  
Leveraging on this result, in Sec. \ref{opt}, we present the cost-optimal resource allocation problem. We use a semidefinite approach to
	%
 formulate our optimization problem for the
case of arbitrary undirected graphs with symmetric weights. Then, we show that this approach can be extended to some kind of not symmetric weighted networks, those whose
adjacency matrix is diagonally symmetrizable. Moreover, for the case of a general
directed weighted graph, we provide a suboptimal solution. 
In Sec \ref{sec:comm} we consider the case when a contact network is shaped 
by an existing community network.  
First, we extend the results in \cite{QEP} (related to equitable partitions in the case of undirected graphs), considering equitable partitions for directed weighted graphs. 
Specifically, we show how a certain kind of structure regularity, in a directed weighted graph, influences the system of differential equations 
that solve for the evolution in time of the approximated infection probability of each node.
Then, we exploit such regularities in graph connectivity for reducing the original dynamical system to a lower dimensional one. By supposing that different curing rates can be chosen depending on the community network structure and that they can be 
optimized for a certain cost function, the latter system is used to reduce the dimension of our optimization problem. 
In the last part of the work, we propose a two-level optimal curing problem, that is, we have
a two-dimensional curing policy, suitable, e.g., when the population can
be divided in two categories (young and elders, male and female, etc,...).
This kind of situation fits well certain networks with equitable partition, such as, e.g., bipartite graphs and interconnected star networks.
In this case we provide a scalable bisection algorithm, that yields an $\varepsilon$-approximation of the optimal solution, in polynomial time in the input size. Finally, we carry out some numerical
experiments.
Proofs not included in previous sections for
better readability are placed in Section \ref{app}.



\section{The epidemic network model}\label{enm}
\stef{In this section, we report some background concepts and new tools that we will use later to find a cost-optimal curing policy}.

Let us consider a SIS epidemic process spreading over a simple 
undirected graph $G=(V,E)$, with edge set 
$E$ and node (vertex) set $V$. The order of $G$, denoted by $N$, is the cardinality of $V$. The edge set of $G$ consists of unordered pairs   $\left\{i,j\right\}$, with $i,j \in V$, and $i \neq j$. Connectivity 
of graph $G$ is conveniently expressed by the symmetric $N \times N$ adjacency matrix $A$.

The viral state of a node $i$, at time $t$, is described 
by a Bernoulli random variable $X_i(t)$, where we set $X_i(t) = 0$, if $i$ is healthy and $X_i(t) = 1$, if $i$ is infected. 
Every node at time $t$ is either infected with probability $p_i(t) = \bP(X_i(t) = 1)$ or healthy (but susceptible) with probability $ 1 - p_i(t)=\bP(X_i(t) = 0) $. \stef{In the homogeneous setting},
the recovery process is a Poisson process with rate $\delta$, and the infection process is a \textsl{per link} 
Poisson process where the infection rate between an healthy and an infected node is $\beta$.
All the infection and recovery processes are independent. 
The SIS process, developing a graph with $N$ nodes, can be modeled
as a continuous-time Markov process with $2^N$ states, covering all possible
combinations in which $N$ nodes can be infected \cite{VanMieghem2009}.
The probability
of the process of being in a certain state can be uniquely determined by the Kolmogorov's differential equations (i.e. a
system of linear differential equations). However, the number
of equations increases exponentially with the number of nodes; this poses
several limitations in order to determines the set of solutions even for small
network order. Hence, often, it is necessary to derive models that are an
approximation of the exact original one \cite{VanMieghem2009,scoglio}.
 
 In this work, we consider a first order mean-field approximation (NIMFA), proposed by Van Mieghem et. al. in \cite{VanMieghem2009}. Basically, NIMFA replaces the original $2^N$ linear differential equations by $N$
non-linear differential equations representing the time-change of the infection probability of each
node.


\textit{Epidemic threshold.}
For a network with finite order $N$, the exact SIS Markov process will always converge
towards its unique absorbing state, that is the zero-state where all nodes are healthy.
\stef{However, the process shows a phase transition behavior: indeed, there is a critical value $\tau_c$ of the effective 
spreading rate $\tau= \beta/\delta$, whereby if $\tau$ is lesser than $\tau_c$, the initial infection
dies out quickly. \fdp{Conversely,} for $\tau$ larger than $\tau_c$, the infection spreading
can last very long in any sufficiently large network \cite{VanMieghem2013, Draief2010, van2016approximate}.
The regime of persistent infection ($\tau > \tau_c$ ), called the metastable or quasi-stationary state, is
reached rapidly given an initial set of infected nodes and can persist for very long time \cite{van2016approximate}.}
\stef{In support of this, numerical simulations of SIS processes reveal that, even for fairly small
networks $(N \simeq 100)$ and when $\tau > \tau_c$, the overall-healthy state is only reached 
after an unrealistically long time. Hence, the indication of the model is that, in the case of real networks, one should 
expect that above the epidemic threshold} the extinction of epidemics is hardly ever attained \cite{VanMieghem2013, Draief2010}. For this reason, the literature is mainly concerned with establishing the value of the epidemic threshold, being a key measure of the robustness against epidemic spreading. \\
In the homogeneous setting, NIMFA determines the epidemic threshold for the effective spreading rate as $\tau^{(1)}_c =\frac{1}{\lambda_{1}(A)}$, 
where \stef{$\lambda_1(A)$ is the spectral radius of the adjacency matrix $A$}, (see \cite{VanMieghem2009,VanMieghem2014}).
When $\tau \leq \tau_c^{(1)}$ the only equilibrium of the NIMFA system is the zero point. When $\tau > \tau^{(1)}_c$, 
there exists
a second non-zero steady-state that reflects well the observed viral behavior \cite{VanMieghem2012a}, and that can be regarded as the analogous of the quasi-stationary state
of the exact stochastic SIS model.  
 NIMFA yields an upper bound for the probability of infection of each node, as well as a lower bound for the epidemic threshold \cite{VanMieghem2009, CatorPositivecorrelations}.
\steff{This fact ensures that $\tau_c^{(1)}$ allows us
to determine a safety region $\left\{\tau \leq \tau^{(1)}_c\right\}$ for the effective spreading rate, that guarantees
the extinction of epidemics in a reasonable time frame. Thus, even though NIMFA is approximated, a design for our optimization problem, based on NIMFA, is always ``safe" or ``secure". }
\subsection{Heterogeneous SIS mean-field model}\label{hetsys}

In this section, we consider a heterogeneous setting. We include the possibility that the infection rate is link specific, denoting by 
$\beta_{ij}$ the infection rate of node $j$ towards node $i$. 
Moreover, each node $i$ recovers at rate $\delta_i$, so that the curing rate is node specific.
 Basically, we allow for the epidemics to spread over a
 {\em directed weighted} graph. 

\stef{A direct weighted graph (or weighted digraph) is a triple $G=(V,E, \rho)$, where the elements of $E$, named \textit{arcs} (or directed edges), are ordered couples $e=(i,j)$ of distinct vertices of $V$, and $\rho:E \rightarrow (0,\infty)$ is a given function; $\rho(e)$ is called the \textit{weight} of $e$.}
\stef{ The matrix $A=(a_{ij})$, with elements $a_{ij}=\rho(i,j)=\beta_{ji}$, is the \textit{weighted adjacency matrix} of $G$.
} In our framework, $e=(i,j)\in E$ and $\rho(e)=\beta_{ji}$ means that node $i$ can infect node $j$ with rate $\beta_{ji}$.
Again, self-loops and multiple edges (multiple arcs with the same direction) are not permitted.
Hereafter, we shall assume that the directed graph is strongly connected, i.e., for all pairs of nodes $i,j \in V$, there is a path form $i$ to $j$ and from $j$ to $i$.

\fdp{As in} the homogeneous case, the SIS model with heterogeneous 
infection and recovery rates is a Markovian process. The time for infected node $j$
to infect any susceptible neighbor $i$ is an exponential random variable with
mean $\beta_{ij}^{-1}$. Also, the time for node $j$ to recover is an exponential random 
variable with mean $\delta_j^{-1}$.   
 A NIMFA model for the heterogeneous setting has been presented first in \cite{Inhom}, where 
a node $i$ can infect all neighbors with the same infection rate $\beta_i$. Here we include the possibility that the infection rates depend on the connection 
between two nodes, thus covering a much more general case.
The NIMFA governing equation for node $i$ in the heterogeneous setting writes as
\begin{equation}
\begin{aligned}\label{het1}
\hskip-1mm \frac{d p_i(t)}{dt}=  \sum_{j=1}^N  \beta_{ij} p_j(t)  -  \sum_{j=1}^N  \beta_{ij} p_i(t) p_j(t)  -\delta_i p_i(t), \; \hskip15mm
i= 1,\ldots,N .
\end{aligned}
\end{equation}

Let the vector $P(t)=(p_1(t), \ldots, p_N(t))^T$ and let $\overline{A}=(\overline{a}_{ij})$ be the matrix defined by $\overline{a}_{ij}=\beta_{ij}$ when $i \neq j$, 
and $\overline{a}_{ii}= - \delta_i$; moreover let $F(P)$ be a column vector whose $i$-th component is   $- \sum_{j=1}^N  \beta_{ij} p_i(t) p_j(t)$. 
Then we can rewrite \eqref{het1} in the following form:
\begin{equation}\label{het}
\frac{d P(t)}{dt}= \overline{A} P(t) + F(P).
\end{equation}
Let 
\[
r(\overline A)= \max_{1 \leq j \leq N} Re(\lambda_j(\overline{A})),
\]
be the \textit{stability modulus}~\cite{Stab} of $\overline A$, where $Re (\lambda_j(\overline{A}))$ denotes the real part of the eigenvalues 
of $\overline A$, $j=1, \ldots, N$. %
\steff{Now, we adopt a result from \cite{Stab} that lead us to find the epidemic threshold, and to extend the global stability analysis of the homogeneous NIMFA system (see e.g. \cite{QEP}) 
to the entirely heterogeneous setting, where each node can potentially infect each of its neighbors with different infection rates. 
We underline that to use the result in \cite{Stab}, the matrix $\overline{A}$ needs to be irreducible, this is equivalent to say that its associated digraph must be strongly connected.}

\begin{theorem}\label{threshhet} 
If  $r(\overline{A}) \leq 0$ then $P=0$ is a globally asymptotically stable equilibrium point in $I_N =[0,1]^N$ for the system \eqref{het}, On the other hand if $r(\overline{A}) > 0$ then there exists a constant solution $P^{\infty} \in I_N -\left\{0\right\}$, such that $P^{\infty}$ is globally asymptotically stable in  $I_N -\left\{0\right\}$ for \eqref{het}.
\end{theorem}
\begin{proof}
See \cite[Thm. 3.1]{Stab}.
\end{proof}

\begin{remark}\label{real1}
Let $A$ be an $N \times N$ irreducible and non negative matrix, $D$ a diagonal matrix with positive entries. Let $\sigma(A-D)$ be the spectrum of the matrix $A-D$,
then the eigenvalue $\lambda \in \sigma(A-D)$, such that $Re(\lambda)=r(A-D)$, is real (this follows also from \eqref{real}). 
\end{remark}

\fdp{The result in Theorem ~\ref{threshhet} is crucial for the cost-optimal curing problem described in the next section. In fact,
  it identifies the value of the epidemic threshold, separating an absorbing phase, where the epidemics will go extinct, from an endemic phase.  Thus, this critical threshold is recognized as a key value for treatment strategies against viral infection. }


\section{\steff{Optimal curing policies for arbitrary weighted networks}}\label{opt}


\stef{Now, leveraging on the result in Theorem \ref{threshhet}, we address the problem of suppressing an epidemic spreading, by a cost-optimal distribution of
  resources within a networked population}.
\stef{Allocating more resources at each node aims to increase
its curing rate, that is reflected, e.g., by
speeding up its detection capabilities and treatments. We consider that recovery resources
have an associated cost that might be different for each node.}
Thus, let us define a cost function which \fdp{measures the expenditure in order to distribute curing resources to all nodes.}
  Let $f_i(\delta_i)$ be a real, linear and monotonically increasing function with respect to $\delta_i$, whose value represents the effort of modifying
the recovery rate of node $i$. 
%



This model fits the case of disease treatment plans: policy makers can distribute different amount of resources (e.g. money for medicines, medical and nursery staff, etc,...) 
in a network of hospitals, or they can design a different health program for different districts, cities, or nations in the case of a timely mass prophylaxis plan.
For instance, in the US, pharmaceutical supply caches and production arrangements have been pre-designated. \fdp{This is done} in order to
 be used for large-scale ongoing prophylaxis and/or vaccination campaigns in case of sudden intentional or natural outbreaks disease \cite{hupert2004community}.

For now, we take into account an arbitrary weighted network. In Sec. \ref{sec:comm}, instead, we shall provide a cost-optimal
  curing policy for a network with community structure.
Hereafter, we consider $f(\delta_i)=c_i \delta_i$, \rev{with $c_i>0$, that we shall call the cost coefficients}, 
 for $i=1,\ldots, N$. Thus, the cost function is the cumulative sum over the nodes' set
\begin{equation*}\label{eq:cost}
U(\Delta)=\sum_{i=1}^N c_i \delta_i,
\end{equation*}
 where $\Delta=(\delta_1, \ldots, \delta_N)$ is the curing rate vector.

\subsection{Undirected graphs with symmetric weights.}
Now, let us assume that $\beta_{ij}=\beta_{ji}$, for all $i,j=1, \ldots, N$, i.e., the weighted 
adjacency matrix \stef{$A=(\beta_{ij})$} is symmetric and, consequently, all its eigenvalues are real. \steff{Basically, now we are considering undirected graphs with symmetric weights.}

Let us define 
the $N \times N$ curing rate matrix, $ D=\diag(\Delta)$. \rev{We remark that, hereafter, we shall indicate with $\lambda_1(A)$ the maximum eigenvalue of $A$.} 
By Theorem \ref{threshhet}, we know that if $\lambda_1\!\big ( A - \diag{( \Delta)} \big ) \leq 0$, then the epidemics will go extinct. 
As we have explained in Section \ref{enm}, the critical threshold for the mean-field model is a lower bound of the threshold of the exact Markov model.
 Thus, the condition $\lambda_1\!\big ( A - \diag{( \Delta)} \big ) \leq 0$ corresponds, in the exact stochastic
 model, to a region where the infectious process dies out exponentially fast for sufficiently large times \cite{van2016approximate}. We recall that, instead, above the exact threshold the overall-healthy state is reached only after an unrealistically long time.
 Hence, in order to find a cost-optimal distribution of resources that guarantees the extinction,
 we seek for the solution of the following problem.
\begin{problem}[Eigenvalue Constraint Formulation]\label{prob:immun1}
Find $\ \Delta \geq 0$ which solves   
\begin{eqnarray*}
&&  {\rm{minimize}} \qquad\qquad U(\Delta)\\
&&\mbox{\rm{subject to:}} \quad\;\; \lambda_1\!\big ( A - \diag{( \Delta)} \big ) \leq 0,  \quad\;\; \Delta \geq 0 \nonumber. 
\end{eqnarray*}
\end{problem}

Problem \ref{prob:immun1} can be reformulated as a semidefinite program, that is a convex optimization problem \cite{boyd96}.
In fact $ \diag(\Delta)=\sum_{i=1}^{N} \Delta_i \diag(\mathbf{e}_i)$, where $\Delta_i$ is the $i$-th component of $ \Delta$ and
$\mathbf{e}_i$ is the $i$-th element of the standard basis so that $\diag(\mathbf{e}_i)\geq 0$. Hereafter, as in \cite{boyd},
the inequality sign in $M \geq 0$, when $M$ is a matrix, means that $M$ is positive semidefinite. Thus, we can express the
optimization problem with eigenvalue constraint as a semidefinite programming (SDP) problem.

\begin{problem}[Semidefinite Programming Formulation]\label{prob:immun2}
Find $\Delta$ which solves   
\begin{eqnarray*}
&&  {\rm{minimize}} \qquad\qquad U(\Delta)\nonumber \\
&&\mbox{\rm{subject to:}} \quad\;\;    \diag{( \Delta) - A} \geq 0\nonumber \\
&&\phantom{\mbox{\rm{subject to:}}} \quad\;\; \Delta \geq 0 \nonumber 
\end{eqnarray*}
\end{problem} 
The feasibility of the problem is always guaranteed, as showed in the following
\begin{theorem}[Feasibility]\label{feas} 
Problem \ref{prob:immun2} is feasible.
\end{theorem}
\begin{proof}
We define $l_{\max}:=\max_i\sum_j a_{ij}$ and choose $\Delta= l_{\max} \mathbf{1}_{N}$, where $\mathbf{1}_{N}$ is the all-one vector of length $N$, consequently, $D=l_{\max}  I_{N}$, with $I_{N}$ identity matrix of order $N$. Then, for any  vector $w=\sum_{i=1}^N z_i v_i$, where $z_i \in \mathbb R$, for $i=1 \ldots N$ and $\left\{v_1, \ldots, v_N\right\}$ is an eigenvector basis of $A$, it holds
\begin{eqnarray*}
&& w^T(A- D)w =  w^T\left( \sum_{i=1}^N   \lambda_i(A)z_iv_i  - l_{\max} w\right) \leq  ( \lambda_1(A) - l_{\max})||w||^2 \leq 0,
\end{eqnarray*}
where the last inequality follows since $\lambda_1(A) \leq \max_i\sum_j a_{ij}$. 
 Hence the chosen vector satisfies the constraint and we can assert that the feasible region is not empty.
\end{proof}

Since the problem is feasible there is always an optimal point on the boundary \cite{boyd96} and, 
by the fundamental result of convex optimization, any locally optimal point of a convex problem is 
globally optimal \cite[Sec. 4.4.2]{boyd}.


\stef{\textit{Existing results.} As introduced in Sec. \ref{intro}, an SDP approach is adopted also in \cite{Pappas} to detect a cost-optimal distribution of protective resources in an arbitrary undirected network. Unlike our approach, they consider that each node $i$ can infect all its neighbors with the same infection
rate $\beta_i$; moreover they describe the minimization of a decreasing vaccination cost function, which depends on the infection rates,
that are allowed to be in a feasible interval. In the second part of the work they propose a greedy approach for the case of all-or-nothing vaccination, i.e., they restrict the infection rate to be in a discrete set, possibly different for each node, $\beta_i \in \left\{\underline{\beta_i}, \overline{\beta_i}\right\}$, where the two values, are fixed a priori.}
 
\stef{With respect to their approach, in our model, each node can potentially infect each of its
neighbors with different infection rates, thus we treat a wider scenario. In addition, from the next section onwards, 
we shall focus, mainly, on a population divided into communities, obtaining a dimensionality reduction of the
    optimization problem \eqref{prob:immun2}. 
		Moreover, in Sec.~\ref{bis} we propose a bisection algorithm for a two-level optimal curing problem, i.e we consider a two-dimension curing policy, providing that the population is divided in two categories, each of which will benefit from one of the two policies. The two available  values of the curing rate are not fixed a priori.}

\stef{At last, in \cite{preciado2014optimal}, Preciado et al., leverage on
  Geometric Programming (GP) techniques for the resource allocation problem, applied to 
arbitrary weighted directed graphs, hence they do not require the symmetry of the adjacency matrix.
	 However, the drawback of such formulation is that it does not fit for a linear cost function of the type we are considering, which is, anyway, a standard cost function of practical relevance.} Thus, in the next section, we show how our formulation of the problem, involving a linear cost function, can be extended to a certain
class of not symmetric matrices. 

\steff{\subsection{Extension to directed weighted networks.}} \label{extens}

The formulation of the optimization problem \eqref{prob:immun2} holds for symmetric weighted adjacency matrix, however we shall show how it can be extended to a certain
class of not symmetric matrices that are diagonally symmetrizable. In this case, for a not
symmetric matrix $A$, there exists a diagonal matrix $G$ such that $G^{-1}A G$ is symmetric,
for similarity their eigenvalues are the same and the semidefinite program formulation can be
applied to $G^{-1}A G$ \footnote{As suggested in ~\cite{wan2008designing}, see \cite{berman1994nonnegative} for a method to check if a matrix is
 symmetrizable, and, in case, the way to chose the diagonal matrix to achieve symmetry.}.
\steff{Thus, we can include also the case of not symmetric weighted adjacency matrix.
A notable example is that of an undirected network where each node $i$ can infect all its neighbors with the
  same infection rate $\beta_i$: the weighted adjacency matrix $BA$ with $A$ symmetric and $B=\diag(\beta_i)$ is
  not symmetric, however it is diagonally symmetrizable, indeed choosing $G=B^{1/2}$ we have
  $B^{-1/2}(BA) B^{1/2}= B^{1/2} A B^{1/2}$, which is symmetric (see, e.g., \cite{Pappas}). Hence, for our problem \ref{prob:immun2},
  we can request that the matrix $B^{1/2} A B^{1/2}$ is semidefinite positive.}

\steff{Otherwise, if we have an arbitrary, strongly connected, directed weighted graph and a not symmetrizable $N$-dimensional adjacency matrix $A$, we can apply our formulation to its Hermitian part, $\cH=(A+A^T)/2$, obtaining a suboptimal solution. More precisely, let $Re\lambda(A)=\left(Re \lambda_1(A), \cdots, Re \lambda_N(A) \right)$ be the vector of the real part of the eigenvalues of $A$ and $\lambda(\cH (A))=(\lambda_1(\cH (A)), \cdots, \lambda_N(\cH (A)))$, both arranged in decreasing order; then it holds 
$Re\lambda(A) \prec \lambda(\cH (A))$ \cite[Thm. 10.28]{zhang2011matrix}, meaning that $\lambda(\cH (A))$ \textit{majorizes} $Re\lambda(A)$ \cite[Sec. 10.1]{zhang2011matrix}}.
\steff{Basically, this result suggests that the stability modulus $\lambda_1(A-D)$ (see Remark \ref{real1}) is smaller than $\lambda_1(\cH(A)-D)$, hence the feasible region of our optimization problem, i.e., where $\lambda_1(\cH(A)-D) \leq 0 $, is a subset of the feasible region for the matrix $A$. Hence, if we solve the problem \eqref{prob:immun2}, choosing $\cH(A)$, 
  the value of the cost function obtained is an upper bound of the cost function that would be sufficient to bring $\lambda_1(A-D)$ at the critical value zero. Thus, we obtain a suboptimal solution, i.e., we will be able to lead the epidemic towards the extinction, but with more effort than it would be sufficient}.\\
\indent\stef{Hence, let us consider a diagonally symmetrizable weighted adjacency matrix $BA$; we want to compare the optimal cost function -- corresponding to the optimal solution of the problem \eqref{prob:immun2}) -- with the suboptimal cost function, obtained considering the hermitian part of $BA$. Besides, we compute also the cost in the case of a uniform curing rate vector for which the maximum eigenvalue of $B^{1/2}AB^{1/2}$ attains zero. We use a standard solver for semidefinite programs (see \cite{SDPT3}).} In Fig. \ref{fig:symm} a)  we consider the cost functions obtained averaging over 300 instances of \ER random graphs with $N=100$ for increasing values of $p$. 
  We take a matrix $B=\diag(\beta_i)$, where the infection rates are generated as uniform random variables \rev{ in the interval $(0.1,3)$, 
    and the $c_i$'s constants in the interval $(0.5,5)$}.
\stef{We observe that the suboptimal cost function is close to the optimal cost function, the closer the lower the
  values of $p$.} \rev{In Fig. \ref{fig:symm} b), instead, we fix the value of $p=0.3$ and we plot the costs as functions of the number of nodes $N$. We can see a growth in the difference between the suboptimal and the optimal cost functions as the number of nodes increase. Ultimately, we obtain always an advantage in the use of suboptimal cost function with respect to the uniform case.}
\vskip 0.15cm
	\steff{In the rest of the paper, we shall consider the case of a network with community structure. We shall show that -- in order to find the epidemic threshold for the system \eqref{het} -- it can be employed a matrix with lower dimension than the starting $N$-dimensional adjacency matrix. In turn, this provides a corresponding reduction in the dimension of our optimization problem (see Sec. \ref{optEq}).} 
\begin{figure*}[t]
  \centering
        \begin{minipage}{0.45\textwidth}
	\includegraphics[width=\textwidth]{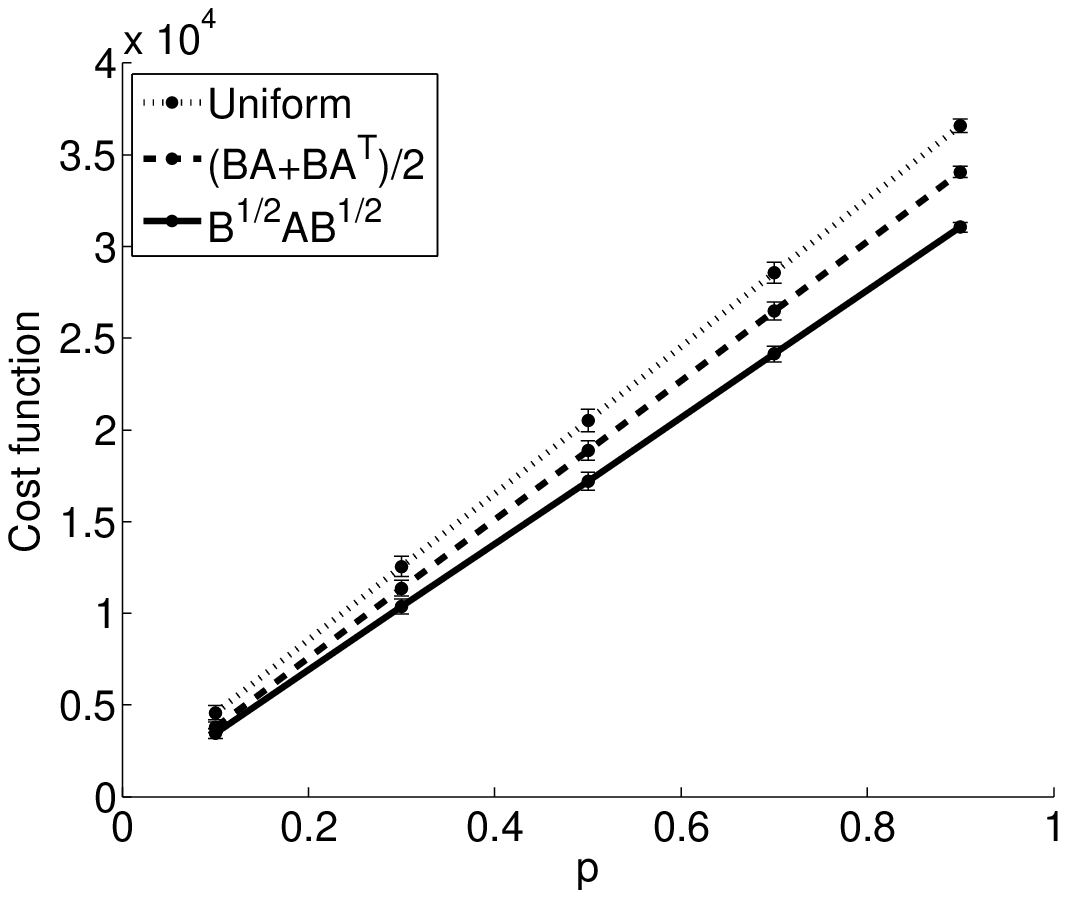}\put(-185,145){a)}
	\end{minipage}
	 \hskip4mm
    \begin{minipage}{0.45\textwidth}
			\includegraphics[width=\textwidth]{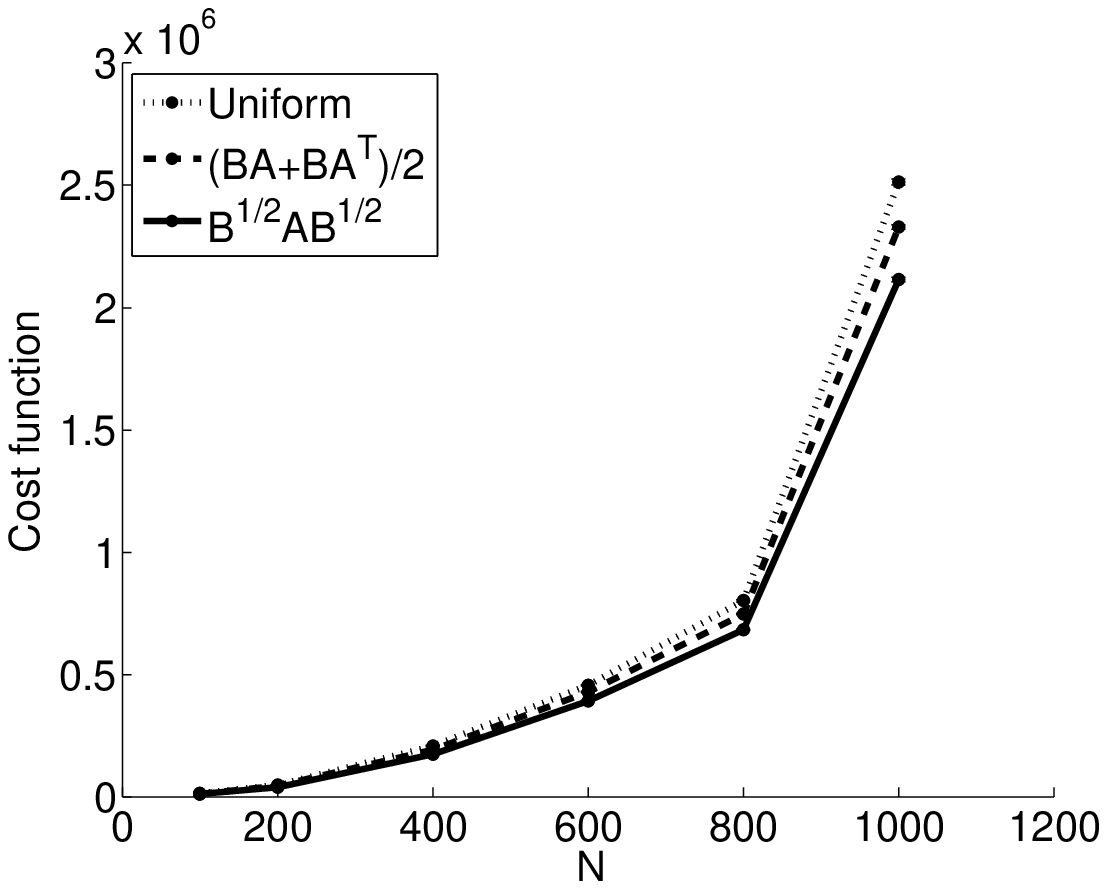}\put(-185,145){b)}
		\end{minipage}
	\caption{\fdpr{Extension to directed weighted networks: comparison between optimal, suboptimal and uniform solutions. a) 
           \rev{the cost values have been obtained averaging over 300 instances of \ER sample graphs with $N=100$, for different $p=0.1,0.3, 0.5,0.7, 0.9$, b) average over $300$ \ER} sample graphs with $N=100,200,400,600,800,1000$ for
            $p=0.3$; $0.95$ confidence intervals are superimposed.} }\label{fig:symm}
\end{figure*}

\section{Community Networks}\label{sec:comm}


 \steff{Hereafter, we shall focus on the case when a contact network is shaped by an existing community network.
This framework captures some of the most salient structural inhomogeneities in contact patterns in many applied contexts \cite{pellis2015seven}}. 
There \fdp{exists an extensive} literature on the effect of network community structures on epidemics. A specific community structure may arise due to, 
for example, geographic separation. Models utilizing this structure are commonly known as ``metapopulation'' models, where the population is compound
of multiple interacting groups, which internally have homogeneous mixing \cite{wan2008designing} (see, e.g., ~\cite{Hanski, Masuda2010}). 
Such models assume that each community shares a common environment or is defined by a specific relationship. 
Some of the most common works on metapopulation regard a population divided into households with two level of mixing 
(\cite{ball1997epidemics, ross2010calculation, ball2002general}). This model typically assumes that contacts, and 
consequently infections, between nodes in the same group occur at a higher rate than those between nodes 
in different groups \cite{pellis2015seven}. Thus, groups can be defined, e.g., in terms of spatial proximity, considering 
that between-group contact rates (and consequently the infection rates) depend in some way on spatial distance, so that,
each individual can be theoretically infected by each of the other individuals.
However, models where infection can only be transmitted by nodes directly connected by an edge, may provide a more realistic
approach to the study of the evolution of the epidemics. In turn, an important challenge is how to consider a realistic
underlying structure and appropriately incorporate the influence of the network topology on the dynamics of epidemic~\cite{pellis2015seven,pellis2012reproduction, 
ball2008network, frank2009, wang2013effect}.
\begin{figure}[t]
	\centering
	\begin{minipage}{0.45\textwidth}
	\includegraphics[width=\textwidth]{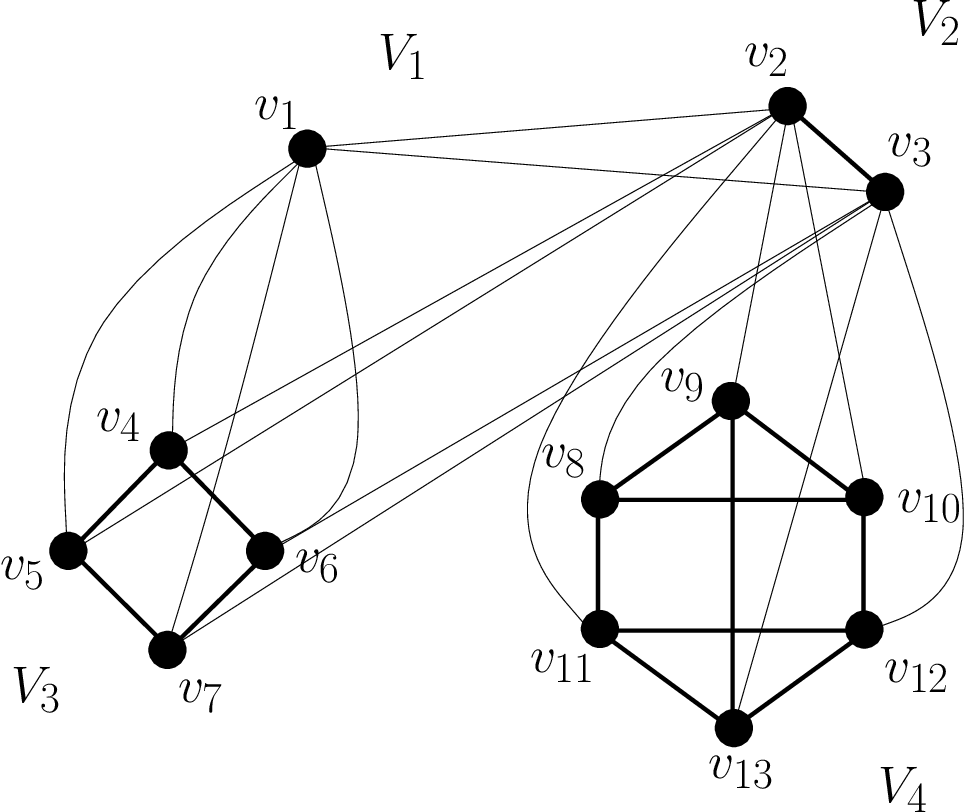}\put(-158,127){a)}
	\end{minipage}
	 \hskip20mm
    \begin{minipage}{0.35\textwidth}
			\includegraphics[width=\textwidth]{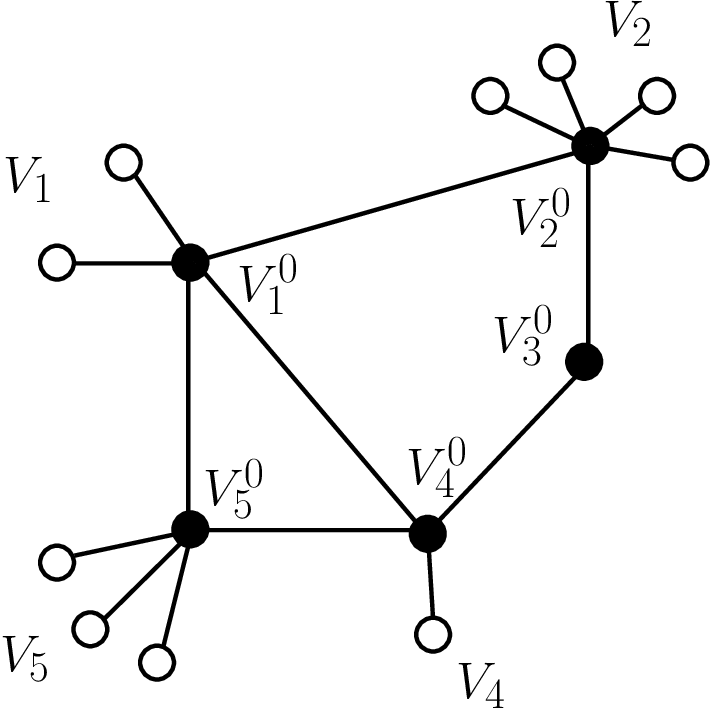}\put(-158,127){b)}
		\end{minipage}
\caption{A sample graphs with equitable partition. 
a) $V=\{V_1, V_2, V_3, V_4\}$, b) Interconnected star networks: $V=\{V_1^0, V_2^0, V_3^0, V_4^0, V_5^0, V_1, V_2, V_3, V_4, V_5\}$}\label{fig:fig11}	
\end{figure}

In \cite{QEP, Bonaccorsi} the authors analyze the dynamics of an epidemics on networks that are partitioned into local
communities, through the first-order mean-field approximation discussed in Section \ref{enm}. The investigation was 
based on the graph-theoretical notion of \emph{equi\-table partition}~\cite{Schwenk,godsil,EvolDelio}. Specifically,
for an undirected graph, let $\pi=\left\{V_1,...,V_n\right\}$ be a partition of the node set $V$, \stef{i.e., a sequence of mutually
disjoint nonempty subsets of $V$, called cells, whose union is $V$, that we assume given a priori};
 $\pi$ is called equitable if the subgraph $G_i$ of $G(V,E)$ induced by $V_i$ is regular for all $i$'s.
Furthermore, for any two subgraphs $G_i$ and $G_j$, whenever there exists at least one connection between nodes in the first 
and second subgraph, then each node in $G_i$ is connected with the same number of nodes in $G_j$. In \cite{QEP,Bonaccorsi} two-level infection rates have been considered: an intra-community infection rate and a lower
inter-community infection rate. In the network structures hereafter, \steff{we generalize the model to more than two levels of infectiousness}. \steff{We observe that usually a community is defined as \fdp{a set of network nodes joined together in tightly-knit groups whereas among such group
connections are looser~\cite{girvan2002community}}. 
Leveraging on the definition of equitable partition, instead, we can also consider that connections between nodes belonging
to the same community can be, eventually, less dense than connections with other communities. Thus, the definition of community acquires a broader sense.}

\steff{Networks with an equitable partition of the node set can describe} models consisting of multiple smaller sub-populations such as, e.g., households, workplaces, 
or classes in a school, when the internal structure of each community is represented by a complete graph (members of a small community
usually 
know each other) and all the nodes of adjacent communities are mutually linked (all member of adjacent communities may potentially 
come into contact). Equitable partitions can be observed also in the architecture of some computer networks where clusters of clients connect to 
single routers, whereas the router network has a connectivity structure with the nodal degree constrained by the number of ports (see as examples Fig. \ref{fig:fig11}b).
\stef{Equitable partitions appear also in the study of synchrony and pattern
formation in coupled cell networks \cite{stewart2003symmetry,golubitsky2005patterns} where they are named “balanced” partitions.
Equitable partitions have been used also to analyze the controllability of multi-agent systems, for the case of a multi-leader setting \cite{rahmani2009controllability}, and for the leader-selection controllability problem, in characterizing the set of nodes from which a given
networked control system (NCS) is controllable/uncontrollable \cite{aguilar2016almost}}.
These works show interesting realistic scenarios for the use of equitable partitions.

In particular, since the size of some real networks might pose limitations in our ability to investigate their spectral properties, we can leverage on the structural regularity of network with equitable partition, to reduce the dimensionality of our optimization problem \eqref{prob:immun2}.

\fdp{Next, we define equitable partitions} for the case of a directed weighted networks, extending the analysis in \cite{QEP} to this framework. %
\steff{With a little abuse of notation, hereafter we shall refer to a partition of a network, to indicate the partition of its node set.}

\subsection{Equitable partitions for weighted directed networks.}\label{sub:eq}
\fdp{The definition of equitable partitions can be extended to weighted directed graphs, based on~\cite[Def. 8.24]{EvolDelio}. That definition %
  \fdp{applies} to oriented weighted graphs \cite[Def. A.1]{EvolDelio}: we prefer to allow for a pair of symmetric oriented edges in order to cover
  naturally unoriented graphs.}
\begin{definition}\label{def:eqpart}
  Let \steff{ $G=(V,E,\rho)$} be a weighted directed graph. The partition $\pi=\left\{V_1,...,V_n\right\}$ of the node set $V$ is called \emph{inward equitable} or \emph{outward equitable} if for all $i,j \in \left\{1, \dots ,n \right\}$, there are  
\begin{align*}
c^{in}_{ij} \in \mathbb{R} \quad \text{s.t.} \quad \sum_{w \in V_j} \rho((v,w)) = c^{in}_{ij}, \quad 
 \hskip5mm  \text{for all} \quad v \in V_i, 
\end{align*}
 or
\begin{align*}
  c^{out}_{ij} \in \mathbb{R} \quad \text{s.t.} \quad \sum_{w \in V_j} \rho((w,v)) = c^{out}_{ij}, 
   \hskip8mm  \text{for all} \quad  v \in V_i, 
\end{align*}
respectively. The partition is called \emph{equitable} if it is both inward and outward equitable, hence for all $i,j \in \left\{1, \dots ,n \right\}$, there are
\begin{align*}
 c_{ij} \in \mathbb{R} \quad \text{s.t.} \quad \sum_{w \in V_j} \rho((v,w)) + \rho((w,v))= c_{ij}, 
   \hskip8mm \text{for all} \quad v \in V_i.
\end{align*} 
\end{definition}
We shall identify the set of all nodes in $V_i$ with the $i$-th {\em community} of the whole population.
\begin{remark}
Let $k_i$ be the number of elements of $V_i$, $i=1, \ldots, n$. If the partition of the node set of a weighted di-graph is equitable, then for all $i,j \in \left\{1, \dots ,n \right\}$,
\begin{equation}\label{out}
k_i c^{out}_{ij}=k_j c^{in}_{ji},
\end{equation}
\end{remark}
An equitable partition generates the \emph{quotient graph} $G/\pi$, which is a \emph{multigraph}, directed and weighted, with cells as vertices. For the sake of explanation,
in the following, we will identify $G/\pi$ with the simple graph having the same vertex
set, i.e. composed by the cells, where an edge exists between two cells, if at least one exists in the original
multigraph.

For the purpose of modeling, nodes of the quotient graph can represent communities, e.g., villages, cities or 
countries. Link weights in the quotient graph, in turn, provide the strength of the contacts between such communities. 
In particular, the weight of a link may be (a non-negative) function of the number of people traveling per day between 
two countries; in fact, the frequency of contacts between them correlates with the propensity of a disease to spread 
between nodes.\\\\

 Related to the quotient graph, there exists a quotient matrix, that contains the relevant structural information of the networks. Thus,
let us consider the $n \times N$ matrix $S=(s_{iv})$, where
\begin{equation*}
s_{iv}=\begin{cases}
\frac{1}{\sqrt{|V_i|}} & \text{ $v \in V_i$}\\
0 & \text{otherwise},
\end{cases}
\end{equation*}
from which it follows that $SS^T=I_n$.
Now let us consider the transpose of the adjacency matrix of the weighted directed graph $G$, that is 
\begin{equation}\label{matb4}
A^T=\overline{A}+D,
\end{equation}
where, \steff{we remember, $\overline{A}$ is the matrix in \eqref{het}} and $D=\diag(\Delta)$ is the curing rate matrix. Then, the transpose 
of the \emph{quotient matrix} of $G$ (with respect to the given partition) is 
\[
Q^T:=SA^TS^T.
\]

We can write the following explicit expression for $Q^T$:
\begin{equation}
Q^T=\diag(c^{out}_{ii})+ \left(\frac{k_i c^{out}_{ij}}{\sqrt{k_ik_j}}\right)_{i,j=1, \ldots, n}
\end{equation}
By \eqref{out}, we can write the matrix $Q^T$ as
\begin{equation}\label{eq:q4}
Q^T=\diag(c^{out}_{ii})+ \left(\sqrt{c^{out}_{ij} c^{in}_{ji}}\right)_{i,j=1, \ldots, n}
\end{equation}

\begin{remark}
We observe that matrix $Q$ in \eqref{eq:q4} might not be symmetric, whereas in the case of undirected 
graphs it is always symmetric (see, e.g., \cite{godsil, QEP}). Even though we have represented the most general 
definition of an equitable partition simpler situations can be represented. E.g., nodes 
of the same community may infect all nodes in another community with the same rate.
\end{remark}

When considering a population partitioned into communities, it may be appropriate to take 
into account the case where all nodes of a tagged community $j$ have the same recovery rate 
$\delta_j$, $j=1, \ldots ,n$. In turn, such rate may differ from one community to the other.
We remember that $N$ is the total number of nodes in the network, whereas $n$ is the number
of communities. 
\begin{definition}
Let us introduce the $1 \times n$ vector of nonzero curing rates $\overline{\Delta}=(\overline{\delta}_1, ... , \overline{\delta}_n)$, that we shall call the reduced curing rate vector and $\overline{D}=\diag(\overline{\Delta})$, the reduced curing rate matrix.
\end{definition}

\steff{Thus, we have the $1 \times N$ curing rates vector $\Delta=(\delta_1, \ldots, \delta_N)$, with components $\delta_z=\overline{\delta}_j$ 
for all $z \in V_j$ and $j=1, \ldots, n$.}

\steff{In appendix \ref{app:dim} we shall discuss when and how it is possible to reduce the original system \eqref{het} to a system of $n$ differential
  equations, through the matrix $Q^T$. Since for our optimization problem the parameter of interest is the epidemic threshold, in this section we limit our self to
  results related to this critical value.}
\begin{lemma}\label{lemma1}
Let $\pi=\left\{V_1, \ldots, V_n\right\}$ be an equitable partition. Let $A^T$ and $Q^T$ be weighted matrices as in \eqref{matb4} and \eqref{eq:q4}, 
respectively. Then, it holds: 

\noindent i) $(A^T-D)S^T=S^T(Q^T-\overline D)$. \\
\noindent ii) For all $\lambda \in \mathbb{C}$ and all $x \in  \mathbb{C}^n$ 
\begin{equation*}
(Q^T-\overline D)x= \lambda x \quad \text{if and only if} \quad  (A^T-D)S^Tx= \lambda S^Tx. 
\end{equation*} 
\end{lemma}

Now let us consider the system of $N$ differential equations \eqref{het}. It is possible to extend \cite[Thm.~4.1]{QEP} to the case of directed graphs. Following
  that result, if we assume that at time $t=0$ the infection probability is equal for all nodes in the same community (while it may differ from one community to the
  other), the number of equations in \eqref{het} can be reduced by using the transpose of the quotient matrix $Q^T$. Hence, the reduced dynamical system writes
\begin{align}\label{eqred}
\frac{d\overline{p}_j(t)}{dt}  &=  (1-\overline{p}_j(t))\sum_{\substack{m=1 \\ m\neq j}}^n  \!\! c^{out}_{jm} \overline{p}_m(t) 
			             + c^{out}_{jj}(1-\overline{p}_j(t))\overline{p}_j(t) - \overline{\delta_j} \overline{p}_j(t),  \hskip15mm j=1,\ldots,n 
\end{align}

where $\overline{p}_j(t)$ is the infection probability of a node in the community $j$.
We can prove that, in the case of a graph whose node set has an equitable partition, and {\em regardless of initial conditions},
the critical threshold for \eqref{het}, applying Thm. \ref{threshhet}, can be determined directly considering the reduced system \eqref{hetQ12}.

\begin{proposition}\label{propthr}
  The elements of the curing rates vector $\Delta=( \delta_1, ...,\delta_{N})$, that determine the critical threshold of \eqref{het}, are identified by the elements
  of $\overline \Delta=(\overline \delta_1, ... , \overline \delta_n)$, in such a way that $\delta_z=\overline \delta_j$ for all $z \in V_j$, $j=1, \ldots, n$, for
  which  
\begin{equation}\label{qthr}
r( Q^T - \overline D) = 0,
\end{equation}
 where $r$ is the stability modulus.
\end{proposition}

Since the quotient matrix and the adjacency matrix have the same stability modulus (and so their transposed do), a computational
advantage can be obtained in the calculation of the critical threshold of the system \eqref{het}.
\steff{This result is very relevant for our optimization problem. Indeed, in the case of a network with equitable partition, we can
  use a lower dimensional matrix to compute the epidemic threshold.} 


\section{Optimization for Networks with Equitable Partitions}\label{optEq}


In this section, we consider a heterogeneous curing control per community. 
\steff{First, we assume that all nodes in community $j$ infect all nodes in community $i$ with the same infection rate,
$\beta_{V_i V_j}$, and that $\beta_{V_i V_j}={\beta}_{V_j V_i}$, $i,j=1, \ldots, n$. In this case, the graph is undirected
and the weights are symmetric, thus the quotient matrix $Q$ is symmetric and has real eigenvalues.}
Now let us consider the $1 \times n$ reduced curing rate vector $\overline \Delta$, 
the cost function writes
\stef{
\[
U(\overline \Delta)=\sum_{j=1}^n c_j k_j \overline{\delta}_j.
\]}

Thus, $U(\overline \Delta)$ is the cost for curing all elements of each community $j$ at rate \stef{$\overline{\delta}_j$}, \steff{where $c_j >0$, $j=1, \ldots,n$}.
We seek for the solution of the following  
\begin{problem}[Eigenvalue Constraint Formulation]\label{prob:immunCOMM1}
Find $\overline \Delta \geq 0$ which solves   
\begin{eqnarray*}
&&  {\rm{minimize}} \qquad\qquad U(\overline \Delta)\\
&&\mbox{\rm{subject to:}} \quad\;\; \lambda_1\!\big (  Q - \diag{(\overline \Delta)} \big ) \leq 0,  \quad\;\; \overline \Delta \geq 0 \nonumber 
\end{eqnarray*}
\end{problem}
which also writes
\begin{problem}[Semidefinite Programming Formulation]\label{prob:immunCOMM2}
Find $\overline \Delta \geq 0$ which solves   
\begin{eqnarray*}
&&  {\rm{minimize}} \qquad\qquad  U(\overline \Delta)\nonumber \\
&&\mbox{\rm{subject to:}} \quad\;\;    \diag{(\overline \Delta) - Q} \geq 0\nonumber \\
&&\phantom{\mbox{\rm{subject to:}}} \quad\;\; \overline \Delta \geq 0 \nonumber 
\end{eqnarray*}
\end{problem} 
Thm.~\ref{feas} guarantees the feasibility of the problem.
 The general case of equitable
  partitions introduced in Sec. \ref{sub:eq},
  may not lead to a symmetric quotient matrix $Q$. However, we may consider suboptimal
  solutions -- as explained in Sec. \ref{extens} -- obtained by applying the formulation of our
  optimization problem to the hermitian part of $Q$.


\steff{In the next section, we consider a simpler version of Problem ~\ref{prob:immunCOMM2}
and we design a more efficient algorithm with respect to the SDP program.}
 

\subsection{Two-level curing problem}\label{sec:two}

\fdpe{The state of the art for SDP solvers such as, e.g., the \texttt{SDPT3} solver used for our numerical computation, provide solutions for moderate size graphs. 
Actually, the  best known bound
  for the 
	complexity of an $\epsilon$-solution attained with an
  interior point method is $O( n^{3.5}\log(1/\epsilon))$, where $\epsilon$ represents the accuracy~\cite{Nemirovski}. }
\stef{The problem can be solved more efficiently when we face a two-level optimal curing problem, for which we shall provide an algorithm that yields an
$\eps$-approximation of the optimal solution, with a complexity equal to $O(\log(n)n^{3.3731} \log(1/\epsilon))$ (see Thm. \ref{thmCompl}).}
\stef{Precisely, we consider only two possible levels of
the nodes local curing rates, let us say $\delta_0$ and $\delta_1$, that are not fixed a priori.
This situation fits well, e.g., in the case of
a network where communities 
  are of ``two types''.}
	Communities of the first type are \stef{eligible for} curing rate $\delta_0$,
  whereas communities of the second type are \stef{eligible for} curing
rate $\delta_1$. For convenience, we define \stef{the former}, {\em central} communities, and \stef{the latter}, {\em terminal} communities.
 
\stef{Such kind of} \fdpe{configuration is suitable for a network that is, e.g., bipartite
(where each node, e.g., represents a full-meshed community), or for an interconnected stars network,
  i.e., a network obtained by interconnecting star graphs by linking stars' central nodes
  (see Fig~\ref{fig:fig11}b).} Let us note that the Barab{\'a}si-Albert graph model \cite{barabasi1999},
that captures the power-law degree distribution often seen (or approximately seen) in real-world networks,
can be regarded as a set of hubs with star graph features \cite{Li2012}.
Bipartite networks, instead, can be used to understand the spreading of sexually
transmitted diseases, in which the population is naturally divided into males and females
and the disease can only be transmitted between nodes of different kinds. Bipartite networks
can also represent the spreading of diseases in hospitals, in which one type of node accounts
for (isolated) patients and the other type for caregivers, or some vector-borne diseases, such
as malaria, in which the transmission can only take place between the vectors and the hosts~\cite{PietSurvey}.

\stef{Thus, let us consider} the following partition of the node set, $\pi_0=\left\{V_1^{0},...,V_m^{0}\right\}$ 
and $\pi_1=\left\{V_1,...,V_{m'}\right\}$. We assume that the node set partition $\pi=\pi_0 \cup \pi_1$
is equitable. 
Let us introduce the curing matrix 
$D=\diag\left(\delta_0 \textbf{1}_{m},  \delta_1 \textbf{1}_{m^{'}}\right)$ and define \[I^0_m
=\begin{bmatrix}
 I_m& 0\\ 
0  & 0
\end{bmatrix}
, \qquad I^1_{m'}
=\begin{bmatrix}
0 & 0\\ 
0  &  I_{m'}
\end{bmatrix}
\]
where $I_m$ is the identity matrix of order $m$. Then, we can write the semidefinite programming for the two-level curing rates, shortly the $2D$ curing problem, as follows:
\begin{problem}[Semidefinite Programming 2D Formulation]\label{prob:twolevel}
Find $\Delta_2=(\delta_0, \delta_1)$ which solves   
\begin{eqnarray*}
&&  {\rm{minimize}} \qquad\qquad U(\Delta_2) \nonumber \\
&&\mbox{\rm{subject to:}} \quad\;\;    {\delta_0 I^0_m + \delta_1 I^1_{m'} -  Q} \geq 0\nonumber \\
&&\phantom{\mbox{\rm{subject to:}}} \quad\;\; \Delta_2 \geq 0 \nonumber 
\end{eqnarray*}
\end{problem} 
The cost function is 

\[
U(\Delta_2)= \sum_{V_j \in \pi_0} k_j f_0(\delta_0) + \sum _{V_z \in \pi_1} k_z f_1(\delta_1),
\] 
where \stef{$f(\delta_0)=c_0 \delta_0$, and  $f_1(\delta_1)= c_1 \delta_1$, with $ c_0,c_1 > 0$, represent the effort to modify the recovery rate for nodes  
belonging to $V_j \in \pi_0$, and $V_z \in \pi_1$, respectively. }

\vskip 0.15cm
In Section \ref{subset:example} of the Appendix, we shall provide some simple examples for the optimal solution of the Problem \eqref{prob:twolevel}.


\subsection{Properties of the $2D$ curing problem}


In the design of our algorithmic solution, we have leveraged on some basic prope\-rties of the $2D$ curing 
problem.  In order to do so, we need a few basic facts recalled next.
\begin{proposition}\label{propr2}
\fdpr{Let $A$ be an $n\times n$ symmetric, irreducible and non negative matrix  and let $D=\diag(\delta_1, ...,\delta_n)$:\\
\noindent i. if $\delta_i=0$ for some $i=1,\ldots n$, then $\lambda_1(A-D) \geq 0$; \\
\noindent ii. The function $(\delta_1, ...,\delta_n) \longmapsto$  \rev{$\lambda_1(A-D)$} is continuous;\\
\rev{\noindent iii. $\lambda_1(A-D)$ is strictly decreasing in $\delta_i$, $i=1,\ldots,n$}.}
\end{proposition}


\fdpr{Let us denote by $\Gamma=\{(\delta_0,\delta_1)|\, \lambda_1(\diag{(\delta_0\mathbf{1}_m,\delta_1 \mathbf{1}_{m'})} - Q )\geq 0\}$ the feasibility region of Prob.~\ref{prob:twolevel}: same argument of Thm.~\ref{feas} let us state that it is non empty; furthermore, the problem structure guarantees $\Gamma$ to be convex \cite{boyd}. We indicate $\Gamma_0$ and $\Gamma_1$ the standard projections of $\Gamma$ onto the $\delta_0$-axis and the $\delta_1$-axis, respectively.}
\begin{lemma}[Monotonicity]\label{lem:decreasing}
Let $\phi:\delta_0\longmapsto \phi(\delta_0)$ be the function that associates to $\delta_0 \in \Gamma_0$ the value $\delta_1=\phi(\delta_0) \in \Gamma_1$ such that \rev{$\lambda_1(  Q -  \diag{(\delta_0\mathbf{1}_m,\delta_1 \mathbf{1}_{m'})} )= 0$}. Then, $\phi$ is decreasing. 
\end{lemma}
\begin{proof}
\fdpr{First, let us show that $\phi$ is a \rev{well defined 
 function over $\Gamma_0$}. Let $\delta_0\in \Gamma_0$, because of feasibility, there exists $\overline \delta_1$, where $(\delta_0,\overline \delta_1)\in\Gamma$, such that \rev{$\lambda_1( Q-\diag{(\delta_0\mathbf{1}_m, \overline \delta_1 \mathbf{1}_{m'})} )\leq 0$}. Furthermore, it holds \rev{$\lambda_1(Q -\diag{(\delta_0\mathbf{1}_m,0\, \mathbf{1}_{m'})})\geq 0$ by $i)$ of Prop. \ref{propr2}}. \rev{Because of $ii)$ and $iii)$ in Prop. \ref{propr2} we know that $\lambda_1(Q-\diag{(\delta_0\mathbf{1}_m,\delta_1 \mathbf{1}_{m'})} )$ is a continuous strictly decreasing function} of $\delta_1$ over $[0,\overline \delta_1]$, so there exists \rev{one and only one value} $\delta_1 \in \Gamma_1$ satisfying the definition of $\phi$.}

Let $z>0$ and assume that $\phi(\delta_0+z)=\phi(\delta_0)+\zeta>\phi(\delta_0)$, for some $\zeta>0$, i.e, 
that $\phi$ is not decreasing.
From the definition of $\phi$ there exists $0 \not = w \in \ker\Big(\diag{(((\delta_0+z)\mathbf{1}_m,\phi(\delta_0+z)\mathbf{1}_{m'})} - Q \Big )$. Hence, we can write

\begin{equation*}
\begin{aligned}\label{eq:unique}
&w^T \Big(  Q  - \diag\big(\delta_0\mathbf{1}_m,\phi(\delta_0)\mathbf{1}_{m'}\big)\Big ) w =  w^T \diag{(z\mathbf{1}_m,\zeta\mathbf{1}_{m'})} w\nonumber + w^T \Big(  Q  - \diag\big((\delta_0+z)\mathbf{1}_m,\phi(\delta_0+z)\mathbf{1}_{m'}\big) \Big ) w \nonumber \\
& =  w^T \diag{((z\mathbf{1}_m,\zeta\mathbf{1}_{m'}))} w > 0,
\end{aligned}
\end{equation*}
where the strict inequality holds because $\diag{(z\mathbf{1}_m,\zeta\mathbf{1}_{m'})} > 0$.
 Since $\lambda_1\Big( Q  - \diag(\delta_0\mathbf{1}_m,\phi(\delta_0)\mathbf{1}_{m'} \Big )=0$, this means that $ Q  - \diag(\delta_0\mathbf{1}_m,\phi(\delta_0)\mathbf{1}_{m'}) $ must be semidefi\-nite  negative and we have a contradiction.  
\end{proof}

\fdpr{We prove next that the search for the optimal solution can be restricted to a compact subset of $\Gamma$.}
\begin{theorem}[Compact search set]\label{thm:compactfeas}
There exist two pairs $(\delta_0^{\min},\delta_0^{\max})$ and $(\delta_1^{\min},\delta_1^{\max})$ such that a solution $\Delta_2^*=(\delta_0^*, \delta_1^*)$ of Prob.~\ref{prob:twolevel} belongs to a compact subset $\Gamma' \subseteq [\delta_0^{\min}, \delta_0^{\max}] \times [\delta_1^{\min}, \delta_1^{\max}] $. 
\end{theorem}
\begin{proof}
Let us define $\hat{c}_0=\sum_{V_j \in \pi_0} c_0 k_j$ and $\hat{c}_1=\sum_{V_z \in \pi_1} c_1 k_z$, 
then we write $U_{l_{\max}}=\hat{c}_0 l_{max} + \hat{c}_1 l_{\max}$, with $l_{max}$ as in  Thm.~\ref{feas}, and 
$U^*=\hat c_0 \delta^*_0 + \hat c_1 \delta^*_1$.
Let us denote $\Delta^{l_{\max}}_2=\left( l_{\max},  l_{\max}\right)$, by Thm.~\ref{feas}, $\Delta_2^{l_{\max}} \in \Gamma$, hence  $U_{l_{\max}}\geq U^*$ and, by defining set 
$\Omega=\{(\delta_0, \delta_1):\hat c_0 \delta_0 + \hat{c_1}\delta_1\leq U_{l_{\max}})\}$, it follows that $(\delta_0^*,\delta_1^*)\in \Gamma'=\Gamma \cap \Omega$; $\Gamma'$  is closed as intersection of closed sets.  

Now, feasibility conditions of Prob. \ref{prob:twolevel} require matrix $ Q- \left(\delta_0 I^0_m + \delta_1 I^1_{m'}\right) $ to be semidefinite negative. We define $f(\delta_0)=\lambda_1\Big( Q - \left(\delta_0 I^0_m + (\frac{U_{l_{\max}}-\hat c_0 \delta_0}{\hat c_1})I^1_{m'}\right) \Big )$: we have $f(l_{\max})\leq 0$ since $( l_{\max},  l_{\max}) \in\Gamma$ and \rev{$f(0)\geq 0$} by $i)$ of Prop. \ref{propr2}.
  By assertion $ii)$ in Prop. \ref{propr2}, $f(\delta_0)$ is a continuous function. Hence, there exists $\delta_0^{\min}$ such that $f(\delta_0^{\min})=0$, and since $\phi$ is decreasing $\phi(\delta_0^{\min})=\delta_1^{\max}$. We can repeat the same reasoning by inverting the role of $\delta_1$ and $\delta_0$ defining 
$g(\delta_1)=\lambda_1\Big( Q - \left((\frac{U_{l_{\max}}-\hat c_1 \delta_1}{\hat c_0}) I^0_m +\delta_1 I^1_{m'}\right)\Big )$. Hence, we can assert that exists $\delta_1^{\min}$ such that $g(\delta_1^{\min})=0$ and $\phi(\delta_1^{\min})=\delta_0^{\max}$.

Finally, by letting $r: \hat c_0\delta_0 + \hat c_1 \delta_1 =  U_{l_{\max}}$, the points $(\delta_0^{\min},\delta_1^{\max})$ and $(\delta_0^{\max},\delta_1^{\min})$  belong to $\partial \Gamma \cap r$, i.e., they belong to $\partial\Gamma'$, so $\Gamma' \subseteq [\delta_0^{\min}, \delta_0^{\max}] \times [\delta_1^{\min}, \delta_1^{\max}]$, and consequently, being $\Gamma'$ closed, it is also compact.
\end{proof}

\begin{remark}\label{rem:compact}
Thm.~\ref{thm:compactfeas} allows us to identify an interval of the values of $\delta_0$ and $\delta_1$ 
 where we can restrict the search of $(\delta_0^*,\delta_1^*)$. Since $\Gamma' \subseteq [\delta_0^{\min}, \delta_0^{\max}] \times [\delta_1^{\min}, \delta_1^{\max}]$ and  $(\delta_0^*, \delta_1^*) \in \Gamma'$, then $\delta_0^* \in [\delta_0^{\min}$, $\delta_0^{\max}]$ and $\delta_1^* \in  [\delta_1^{\min}, \delta_1^{\max}]$. 
 This is one key property in the algorithmic search of the optimal solution presented in the following section.
\end{remark}

Finally, a direct proof that the optimal solution lies on $\partial \Gamma'$ follows:
\begin{corollary}\label{cor:feasborder}
A solution $\Delta_2^*=(\delta_0^*, \delta_1^*)$ of Prob.~\ref{prob:twolevel} belongs to $\partial \Gamma' \cap \Omega$.
\end{corollary}
\begin{proof}
Let us assume $\Delta_2^*=(\delta_0^*,\delta_1^*) \in \Gamma' \setminus \partial \Gamma'$. 
$\Delta_2^*$ is feasible, hence $\lambda_1( Q - D)<0$, with $D=\diag\left(\delta^*_0 \textbf{1}_{m},  \delta^*_1 \textbf{1}_{m'}\right)$. From Prop.~\ref{propr2}, again we can find $0<\delta_1'<\delta_1^*$ such that  $\lambda_1( Q - \diag(\delta_0^*\textbf{1}_{m},\delta_1'\textbf{1}_{m'}))=0$, where, i.e., $\Delta_2'=(\delta_0^*,\delta_1')\in \partial \Gamma'$. But, $U(\Delta_2^*)-U(\Delta_2')=\hat c_1(\delta_1^*-\delta_1')>0$. Contradiction.
\end{proof} 


\subsection{Bisection Algorithm}\label{bis}


Tab.~\ref{algo1} reports on the pseudocode of algorithm \texttt{OptimalThreshold2D}: it solves 
the $2D$ curing problem. It employs three additional functions~\texttt{LeftCorner} 
(Tab.~\ref{algo2}) ~\texttt{RightCorner} and, finally, \texttt{BisectionThreshold} (Tab.~\ref{algo3}). 

\texttt{LeftCorner} identifies via bisection feasible point $(\delta_0^{\min}, \delta_1^{\max})$; the
bisection search operated by \texttt{LeftCorner} -- see proof of Thm.~\ref{thm:compactfeas} -- is performed
along values $\delta_1=f(\delta_0)$. The companion function \texttt{RightCorner} identifies the point
$(\delta_0^{\max}, \delta_1^{\min})$; the pseudocode is omitted for the sake of space. 

Procedure \texttt{isNegativeDefinite} is the standard test for a real symmetric matrix $A$ 
to be negative definite; it requires to verify $\sgn \left(\det(A_k)\right)=(-1)^k$ where $A_k$ is the $k$-th principal 
minor of $A$, i.e., the matrix obtained considering the first $k$ rows and columns only. 
Finally, the~\texttt{OptimalThreshold2D} algorithm performs a bisection search based on a 
subgradient descent over the utility function $U(\delta_0)=\hat c_0 \delta_0 + \hat c_1 \phi(\delta_0)$.

\begin{remark}\label{rem:subg}
In Tab.~\ref{algo1} we have reported an implementation assuming the calculation of the 
subgradient $\partial U$ at each mid point $x$. However, it is sufficient to evaluate the increment 
at a point $x+\epsilon_1$ within the feasibility region for some $\epsilon_1>0$: if $U(x)<U(x+\epsilon_1)$, 
then, due to convexity, the whole interval $[x+\epsilon_1,+\infty)$ can be discarded. 
Conversely, if $U(x)>U(x+\epsilon_1)$, then, due to convexity, the whole interval $[0,x)$ 
can be discarded during the search. This operation can be performed at a cost $O(1)$ when $U(x)$
and $U(x+\epsilon_1)$ are known, i.e., at the cost of two calls of \texttt{BisectionThreshold}.
\end{remark}
We note that  \REPEAT ~loop stops when $\epsilon>\prod |\lambda_i|=|\det( Q - D)|>|\lambda_1|^n$, i.e., when
$|\lambda_1|< (\epsilon)^{1/n}$. Furthermore, the termination condition in \texttt{BisectionThreshold},
\texttt{LeftCorner} and \texttt{RightCorner} requires $\Delta_2$ to lie within the feasible region {\em and}
the determinant to be smaller than $\epsilon$. 

\begin{table}[t]
\caption{\texttt{OptimalThreshold2D}: solves the $2D$ optimal curing problem via the bisection search.}
\begin{tabular}{|p{0.89\columnwidth}|}
\hline
$(\delta_0^*,\delta_1^*)\;$=$\;$\texttt{OptimalThreshold2D}$(Q,c_0,c_1)$\\\hline
Receives: $Q$, $c_0$, $c_1$\\
Returns: $\delta_0^*$, $\delta_1^*$\\
Initialize:  $\, (\delta_l,\delta_1^{\max}) =$ \texttt{LeftCorner}$(Q,c_0,c_1)$ \\
$\qquad\qquad$             $(\delta_1^{\min},\delta_r) =$ \texttt{RightCorner}$(Q,c_1,c_0)$ \\
$\qquad\qquad$ $k\leftarrow 1$, $U_{k-1}\leftarrow 0$, $U_k \leftarrow \infty$ \\
1: $\quad$ \WHILE~$|U_{k}-U_{k-1}|>\epsilon$\\
2: $\quad$ $\quad$ $\delta_0^*=(\delta_{l}+\delta_{r})/2$\\
3: $\quad$ $\quad$ $\delta_1^* \leftarrow$ \texttt{BisectionThreshold}$(Q,\delta_0^*)$ \\
4: $\quad$ $\quad$ $U_{k+1}=\hat{c}_0 \, \delta_0^* + \hat{c}_1 \, \delta_1^*$ \\
5: $\quad$ $\quad$  \IF~ $\partial U_k<0$ \hskip5mm \% (see Rem.~\ref{rem:subg})\\
6: $\quad$ $\quad$  $\quad$ \THEN $\quad$ $\delta_{r}=\delta_0^*$\\
7: $\quad$ $\quad$  $\quad$ \ELSE $\quad$ $\delta_{l}=\delta_0^*$\\
8: $\quad$ $\quad$  \END\\
9: $\quad$ $\quad$  $k\leftarrow k+1$\\ 
9: $\quad$ \END\\
\hline
\end{tabular}
\label{algo1}
\end{table}
\begin{table}[t]
\caption{\texttt{LeftCorner}: identifies the left corner of $\Gamma' \subseteq \Gamma$ (Thm.~\ref{thm:compactfeas}); 
the pseudocode of the dual function $(\delta_0^{\max},\delta_1^{\min}) =$ \texttt{RightCorner}$(Q,c_0,c_1)$ 
is omitted for the sake of space.}
\begin{tabular}{|p{0.93\columnwidth}|}
\hline
$(\delta_0^{\min},\delta_1^{\max}) =$ \texttt{LeftCorner}$(Q,c_0,c_1)$ \\
\hline
Receives: $Q$, $c_0$, $c_1$\\
Returns: $\delta_0^{\min}$ \\
Initialize:  $U_{\max} \leftarrow  (\hat c _0+ \hat c_1)l_{\max}$ \\1: $\;$ \REPEAT \\
2: $\quad$ $\quad$ $\delta_0^{\min}=(\delta_{l}+\delta_{r})/2$\\
3: $\quad$ $\quad$ $\delta_1^{\max} \leftarrow$ $\frac{U_{\max} - \hat c_0 \delta_0^{\min}}{c_1}$ \\
4: $\quad$ $\quad$ $D=\diag(\delta_0^{\min}\text bf{1}_{m},\delta_1^{\max}\textbf{1}_{m'})$\\
5: $\quad$ $\quad$ $X$ $\leftarrow$ \texttt{isNegativeDefinite}$(Q-D)$ \\
6: $\quad$ $\quad$  \IF $X = \texttt{true}$ \\
7: $\quad$ $\quad$  $\quad$ \THEN $\quad$ $\delta_{r}=\delta_0^*\quad$ \% discard larger values \\
8: $\quad$ $\quad$  $\quad$ \ELSE $\quad$ $\delta_{l}=\delta_0^*\quad$ \% discard smaller values\\
9: $\quad$ $\quad$  \END\\
10: $\quad$ $\quad$  $T=\det(Q -D)$ \\
12: $\;$ \UNTIL  \quad $X==$\TRUE \quad \AND \quad $|T|<\epsilon$\quad \% Termination condition \\
\hline
\end{tabular}
\label{algo2}
\end{table}
\begin{table}[t]
\caption{\texttt{BisectionThreshold}: given feasible $\delta_0$, finds $\delta_1$ such that $(\delta_0,\delta_1)$ lies on the frontier of the feasibility region.}
\begin{tabular}{|p{0.93\columnwidth}|}
\hline
$\quad$ $\quad$ $\delta_1 =$ \texttt{BisectionThreshold}$(Q,\delta_0)$ \\
\hline
Receives: $Q$, $\delta_0$\\
Returns: $\delta_1$\\
Initialize: $T$ $\leftarrow \inf$, $\delta_{l}=0$, $\delta_{r}\leftarrow  \max_i\sum_j a_{ij}$ \\
1: $\;$  \REPEAT\\
2: $\quad$ $\quad$ $\delta_1=(\delta_{l}+\delta_{r})/2$\\
3: $\quad$ $\quad$ $D$ $\leftarrow$ $\diag(\delta_0\textbf{1}_{m},\delta_1\textbf{1}_{m'})$\\
4: $\quad$ $\quad$ $X$ $\leftarrow$ \texttt{isNegativeDefinite}$(Q-D)$ \\
5: $\quad$ $\quad$  \IF $\;X = \texttt{true}$ \\
6: $\quad$ $\quad$  $\quad$ \THEN $\quad$ $\delta_{r}=\delta_1$ \% discard larger values \\
7: $\quad$ $\quad$  $\quad$ \ELSE $\quad$ $\delta_{l}=\delta_1$ \% discard smaller values\\
8: $\quad$ $\quad$  \END\\
9: $\quad$ $\quad$  $T=\det( Q -D)$\\
10: $\;$  \UNTIL \quad $X==$\TRUE \quad \AND \quad $|T|<\epsilon$\quad \% Termination condition \\
\hline
\end{tabular}
\label{algo3}
\end{table}
\begin{theorem}[Correctness]
\texttt{OptimalThreshold2D} is an $\epsilon$-approximation of an optimal solution.
\end{theorem}
\begin{proof}
The algorithm operates a bisection search for a global minimum of 
$U(\Delta_2)=\hat c_1 \delta_0 + \hat c_1 \phi(\delta_0)$, where 
$U(\Delta_2)$ is a convex function. %
\fdpr{Let $V=U_{l_{\max}}$ and $\Delta_2^*$ be the optimal solution:} from the properties
of the bisection search on (quasi-)convex functions \cite{boyd}[Ch. 4, pp. 145], 
the accuracy at step $r= \lceil\log_2 ( V/\epsilon ) \rceil $ of the algorithm 
is \fdpr{$|U_r - U(\Delta_2^*)|< V\, 2^{-r} < \epsilon$.}
\end{proof}
\vskip 0.12cm
Furthermore, we can characterize the computational complexity of the algorithm.
\begin{theorem}[Complexity]\label{thmCompl}
\steff{The time complexity of \texttt{OptimalThreshold2D} is $O(n^{1+\ell} \log_2(n/\epsilon) )$}
where $\ell=2.373$.
\end{theorem}
\begin{proof}
\steff{The number of iterations of the bisection search \WHILE~loop (lines $1$ to $9$ in Tab.~\ref{algo1}) is
$O(\log_2(n/\epsilon))$.} This follows again from elementary properties of bisection 
search  \cite{boyd}[Ch. 4, pp. 145]. In fact, the bisection search
operates for $0\leq U(\delta_0) \leq U_{l_{\max}}$ and $U_{l_{\max}}=l_{\max}(\hat{c}_0+\hat{c}_1)$. Finally, indeed, 
$l_{\max}\leq (n-1) \max_{i,j} q_{ij}$. 

Same argument on the measure of the search intervals of \texttt{Bisection\-Threshold},
\texttt{LeftCorner} and \texttt{RightCorner} let us conclude that they require \steff{$O(\log_2(n/\epsilon))$}
iterations of the~\REPEAT ~loop. 

Finally, test \texttt{isNegativeDefinite} appearing in \texttt{Threshold2D}, \texttt{LeftCorner} 
and \texttt{Right-} \texttt{Corner} requires the computation of $n-1$ determinants of the principal minors of
$A - D$ at cost $O(n^{1+\ell})$. Here $\ell$ is the exponent for fast matrix 
multiplication~\cite{Aho_DesiCompAlgos}. In the case of the Coppersmith-Winograd algorithm 
for fast matrix multiplication it holds $\ell=2.373$. 
\end{proof}


\subsection{Numerical Results}\label{num}


In this section we present the results of numerical experiments in the case of interconnected 
stars networks, a sample network is depicted in Fig.~\ref{fig:fig11}b).  
In Figure~\ref{ratio}a) we compare the ratio between the cost $U_u$ of the uniform 
curing rate vector, and the optimal cost  $U^*=U(\Delta^*)$ obtained by solving the $2D$ curing 
Prob. \ref{prob:twolevel}, \rev{by means of the
\texttt{OptimalThreshold2D}}. The uniform curing rate vector is $\Delta= \delta {\mathbf 1}_N$, 
where $\delta$ is the value such that the threshold in~\eqref{qthr} is attained. For this experiment, we consider that the 
infection spreads with rate $\beta_{V_i^0 V_j^0}=\beta_0$ among the central communities and with rate $\beta_{V_i V_i^0}=\beta_1$ between a central node and a node in its adjacent terminal community, moreover we assume $c_0=c_1=1$.
We consider that each terminal community has the same number of elements $k$. The 
computation is made for different values of $k$, for three different sample networks, with $m=50$ central nodes and $m'=50$ terminal communities. Sample networks differ for the average degree of the central nodes. Central nodes are 
connected as \ER random graphs with $p=0.2$, $p=0.3$, $p=0.6$, respectively.
The plot confirms that a larger gain is obtained, \stef{in terms of costs}, by $2D$ curing policy versus a
uniform approach, in particular, the larger the denser the network, namely, for larger $p$ in our samples.
For the interconnected stars networks samples, in particular, we observe one order of magnitude gain in
the cost function. We see that the advantage increases as the number of elements $k$ increases, with a
$\sqrt k$ shaped ratio \eqref{eq:ratio} as derived in closed form  for the case $m=m'=2$. In Figure~\ref{ratio}b)
we  have instead reported, only, on the optimal cost  $U^*$ for different values of $c_0$ and $c_1$. In
particular, we observe that the optimal cost appears to depend linearly on the community size $k$.
Larger costs are incurred in the case when \rev{the coefficient $c_0$, related to the expenditure for the central nodes,} 
is larger than $c_1$. 
This is in line with the fact that central communities are more connected than terminal 
communities, and consequently we need for more investments in such a way that the infection is kept 
subcritical.
\begin{figure*}	
\centering 
\begin{minipage}{.47\textwidth}
\centering
\includegraphics[width=\textwidth]{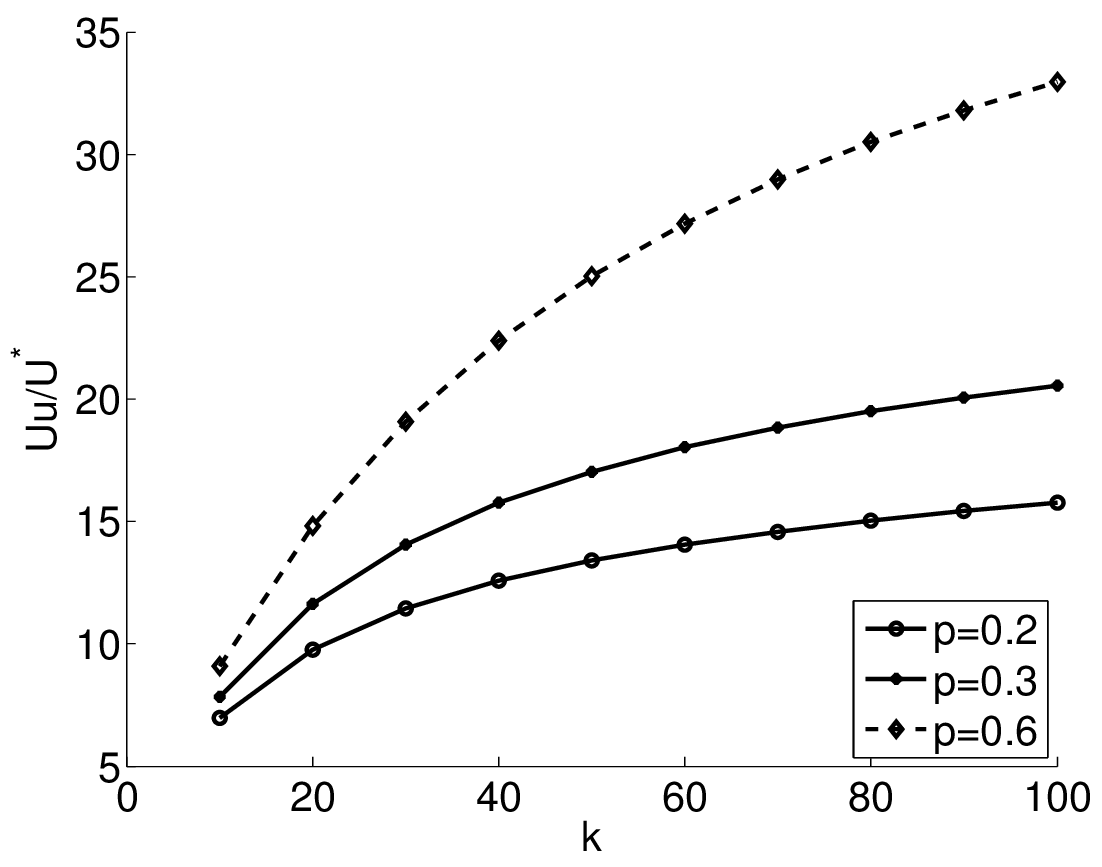}
\put(-195,145){a)}
\end{minipage}
\begin{minipage}{.47\textwidth}
\centering
\includegraphics[width=\textwidth]{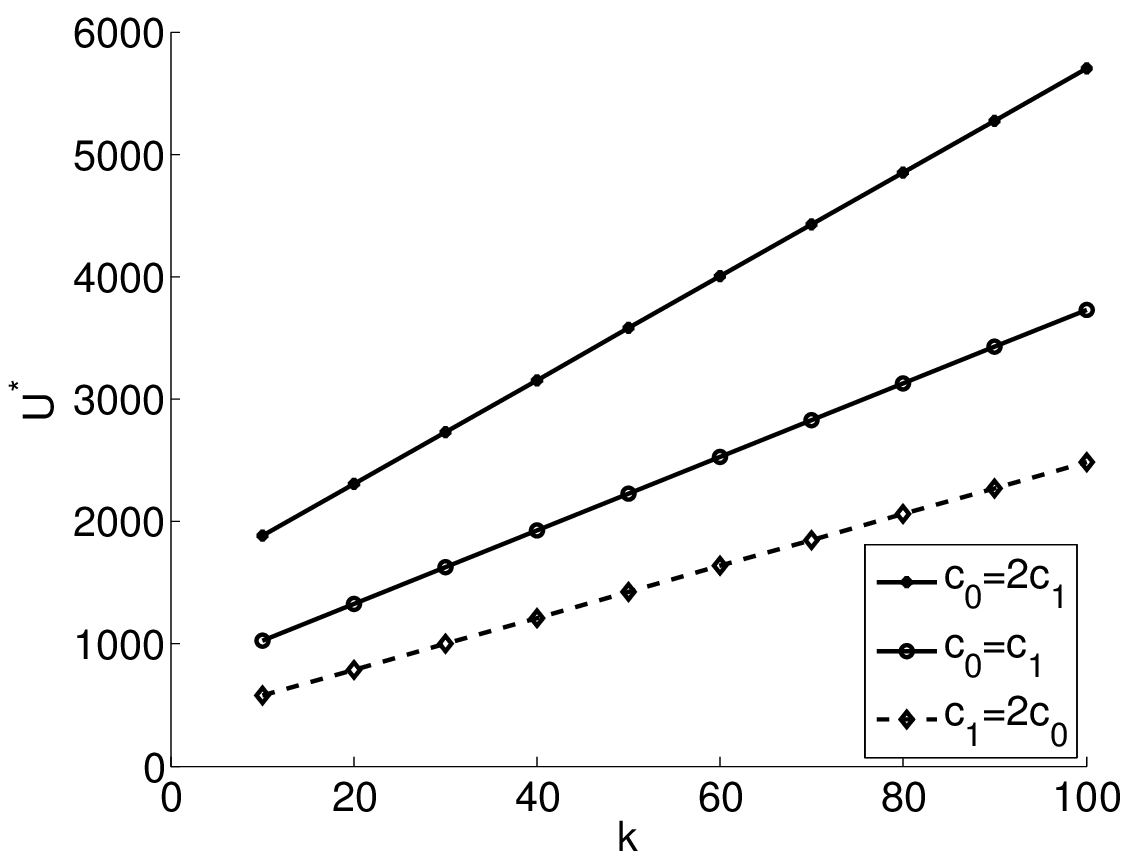}
\put(-195,145){b)}
\end{minipage}%
\caption{(a) \rev{Ratio $U_u/U^*$ for increasing size $k$ of the terminal communities of interconnected star networks with $m=m'=50$. The three curves refer to networks where the central nodes are connected
as \ER graphs generated for $p=0.2$, $p=0.3$ and $p=0.6$ respectively; $\beta_0=1$, $\beta_1=0.3$, $c_0=c_1=1$.}
  (b) Cost function $U^*$ for increasing dimension $k$ of the terminal communities, $\beta_0=1$, $\beta_1=0.3$. The curves refer to the case $p=0.3$ in the cases $c_0=2c_1$, $c_0=c_1$ and $2 c_0=c_1$, respectively.}\label{ratio}
\end{figure*}

In Tab.~\ref{table:perfhom} we compare the performance of \texttt{OptimalThreshold2D} with an SDP solver, namely the \texttt{SDPT3} solver \cite{SDPT3}.
The \texttt{SDPT3} solver generates a solution using a primal-dual interior-point algorithm which leverages on the infeasible path-following paradigm.
As reported in Tab.~\ref{table:perfhom}, when the solver is applied to Prob.~\ref{prob:twolevel}, we denote the corresponding solution as \texttt{SDPT3}~\texttt{(2D)}. For
the sake of comparison, we have reported also on the optimal solution derived with the same solver when curing rates are optimized per node (Prob ~\ref{prob:immun2}),
and we refer to this solution as \texttt{SDPT3}. The solution is provided on a graph with $m=m'=50$ and $c_0=c_1=1$, for increasing values of the terminal
community dimension $k$. We can observe that for the interconnected stars network, in the case of two infection rate levels, \texttt{SDPT3}, ~\texttt{SDPT3}~\texttt{(2D)}
and \texttt{Optimal\-Threshold2D} provide similar values. This result suggests that, in this particular case, there is no advantage to treat each node with different
curing policies: the general solution obtained using \texttt{SDPT3} has \fdpr{similar} performance as the one generated solving the $2D$ formulation of the problem,
i.e., by using \texttt{OptimalThreshold2D} or \texttt{SDPT3}~\texttt{(2D)} in the $2D$ case. We observe also that the curing rate of terminal nodes appears insensitive
to the increase of the terminal communities size $k$, as it can be seen, with a direct computation, in the examples \ref{subset:example} for the case with
$m=1$ and $m=2$, respectively. In Tab.~\ref{table:perfhet}, same performance evaluation has been reported for the same sample graph and \rev{the same cost coefficients}, but studying the case with more than two infection rate levels; specifically a central node can eventually infect each of its adjacent central
nodes with a different infection rate, also the infection rate between a central node and a terminal community can vary from a subgraph to another. The infection rates
are generated as uniform random variables \rev{in the interval $(0.1,1.9)$ and $(0.03,0.57)$ for the speed of infection between central communities and between a central node and a terminal community respectively.}

\begin{table*}
\caption{Performance of \texttt{OptimalThreshold2D}, \texttt{SDPT3~(2D)} and \texttt{SDPT3}. \rev{The graph considered is an interconnected star network with $m=m'=50$, where the connection between the central nodes} are represented by a \ER graph generated with $p=0.2$; $c_0=c_1=1$ and \rev{the values of the weights are $\beta_0=1$, and $\beta_1=0.3$.} \rev{In the \texttt{SDPT3} case, the values of $\delta_i^*$, $i=0,1$, represent the averaged value of the node-specific curing rates over a community.}}
\centering 
\begin{scriptsize}
\begin{tabular}{|p{0.1cm} | p{0.1cm}p{0.2cm}p{0.5cm} | p{0.1cm}p{0.2cm}p{0.5cm} | p{0.1cm}p{0.2cm}p{0.5cm} |}    
\hline\hline   
    & \multicolumn{3}{c|}{\texttt{OptimalThreshold2D}}  &  \multicolumn{3}{c|}{\texttt{SDPT3~(2D)}} &  \multicolumn{3}{c|}{\texttt{SDPT3}}\\ \hline
 k  & $U^*(10^3)$ & $\delta_0^*$ & $k \cdot \delta_1^*$ & $U^*(10^3)$ & $ \delta_0^*$ & $k \cdot \delta_1^*$ 
 & $U^*(10^3)$ & $\delta_0^*$ & $k \cdot \delta_1^*$  \\ \hline\hline
10  &  0.8057  &  13.116  &  0.29979  &  0.80572  &  13.1144  &  0.3  &  0.772  &  12.44  &  0.3\\ \hline
20  &  1.1057  &  16.1163  &  0.29984  &  1.1057  &  16.1144  &  0.3  &  1.072  &  15.44  &  0.3\\ \hline
30  &  1.4056  &  19.1222  &  0.29965  &  1.4057  &  19.1145  &  0.3  &  1.372  &  18.44  &  0.3\\ \hline
40  &  1.7055  &  22.1294  &  0.29951  &  1.7057  &  22.1144  &  0.3  &  1.672  &  21.44  &  0.3\\ \hline
50  &  2.0053  &  25.1354  &  0.29943  &  2.0057  &  25.1144  &  0.3  &  1.972  &  24.44  &  0.3\\ \hline
60  &  2.3052  &  28.1414  &  0.29936  &  2.3057  &  28.1144  &  0.3  &  2.272  &  27.44  &  0.3\\ \hline
70  &  2.6049  &  31.1578  &  0.29916  &  2.6057  &  31.1145  &  0.3  &  2.572  &  30.44  &  0.3\\ \hline
80  &  2.9047  &  34.1792  &  0.29893  &  2.9057  &  34.1144  &  0.3  &  2.872  &  33.4401  &  0.3\\ \hline
90  &  3.2043  &  37.1929  &  0.29882  &  3.2057  &  37.1144  &  0.3  &  3.172  &  36.4402  &  0.3\\ \hline
100  &  3.5039  &  40.1319  &  0.29947  &  3.5057  &  40.1144  &  0.3  &  3.472  &  39.44  &  0.3\\ \hline
\end{tabular}\label{table:perfhom}
\end{scriptsize}
\end{table*}
\begin{table*}
\caption{Performance of \texttt{OptimalThreshold2D}, \texttt{SDPT3~(2D)} and \texttt{SDPT3}, sample graphs are obtained from the same graph used in Tab.~\ref{table:perfhom} and $c_0=c_1=1$. The infection rates
are generated as uniform random variables in the interval $(0.1,1.9)$ and $(0.03,0.57)$, for rates between central communities and between a central node and a terminal community respectively.  \rev{In the \texttt{SDPT3} case, the values of $\delta_i^*$, $i=0,1$, represent the averaged value of the node-specific curing rates over a community.}}
\centering 
\begin{scriptsize}
\begin{tabular}{|p{0.1cm} | p{0.15cm}p{0.2cm}p{0.60cm} | p{0.15cm}p{0.2cm}p{0.60cm} | p{0.15cm}p{0.2cm}p{0.60cm} |}    
\hline\hline   
    & \multicolumn{3}{c|}{\texttt{OptimalThreshold2D}}  &  \multicolumn{3}{c|}{\texttt{SDPT3~(2D)}} &  \multicolumn{3}{c|}{\texttt{SDPT3}}\\ \hline
 k  & $U^*(10^3)$ & $\delta_0^*$ & $k \cdot \delta_1^*$ & $U^*(10^3)$ & $ \delta_0^*$ & $k \cdot \delta_1^*$ 
 & $U^*(10^3)$ & $\delta_0^*$ & $k \cdot \delta_1^*$  \\ \hline\hline
10  &  1.9963  &  34.1309  &  0.57945  &  1.9963  &  34.1309  &  0.57945  &  1.9379  &  33.4102  &  0.53481\\ \hline
20  &  2.5603  &  39.3993  &  0.5903  &  2.5603  &  39.3993  &  0.5903  &  2.4307  &  38.3384  &  0.51382\\ \hline
30  &  3.1665  &  44.9001  &  0.61431  &  3.1666  &  44.9001  &  0.61431  &  2.9536  &  43.5669  &  0.51682\\ \hline
40  &  3.7954  &  50.6431  &  0.63164  &  3.7957  &  50.6431  &  0.63164  &  3.4781  &  48.8125  &  0.51876\\ \hline
50  &  4.4332  &  56.4137  &  0.64499  &  4.4336  &  56.4137  &  0.64499  &  4.0279  &  54.3098  &  0.52495\\ \hline
60  &  5.1272  &  62.627  &  0.66528  &  5.1278  &  62.627  &  0.66528  &  4.6049  &  60.0804  &  0.53364\\ \hline
70  &  5.8175  &  69.0209  &  0.67612  &  5.8184  &  69.0209  &  0.67612  &  5.1595  &  65.6263  &  0.53663\\ \hline
80  &  6.4779  &  74.8633  &  0.68369  &  6.4792  &  74.8633  &  0.68369  &  5.6708  &  70.7388  &  0.53346\\ \hline
90  &  7.1277  &  80.4688  &  0.68984  &  7.1298  &  80.4688  &  0.68984  &  6.1686  &  75.717  &  0.5295\\ \hline
100  &  7.8361  &  87.1312  &  0.6959  &  7.839  &  87.1312  &  0.6959  &  6.7182  &  81.2132  &  0.53151\\ \hline
\end{tabular}
\label{table:perfhet}
\end{scriptsize}
\end{table*}
As seen there, by curing nodes with different curing policies it is possible to attain lower costs at larger values of 
$k$. This effect is depicted also in Fig.~\ref{fig:perf}. In particular, again, the $2D$ curing rates output 
of \texttt{SDPT3~(2D)} and \texttt{OptimalThreshold2D} show similar performance and the optimal curing rate 
of terminal communities appears insensitive to the increase of the terminal communities size. \fdpr{We observe that 
  the relative advantage of the \texttt{SDPT3} tends to increase with the size of the terminal communities.}
\begin{figure}[t]
\centering
\includegraphics[width=0.45\textwidth]{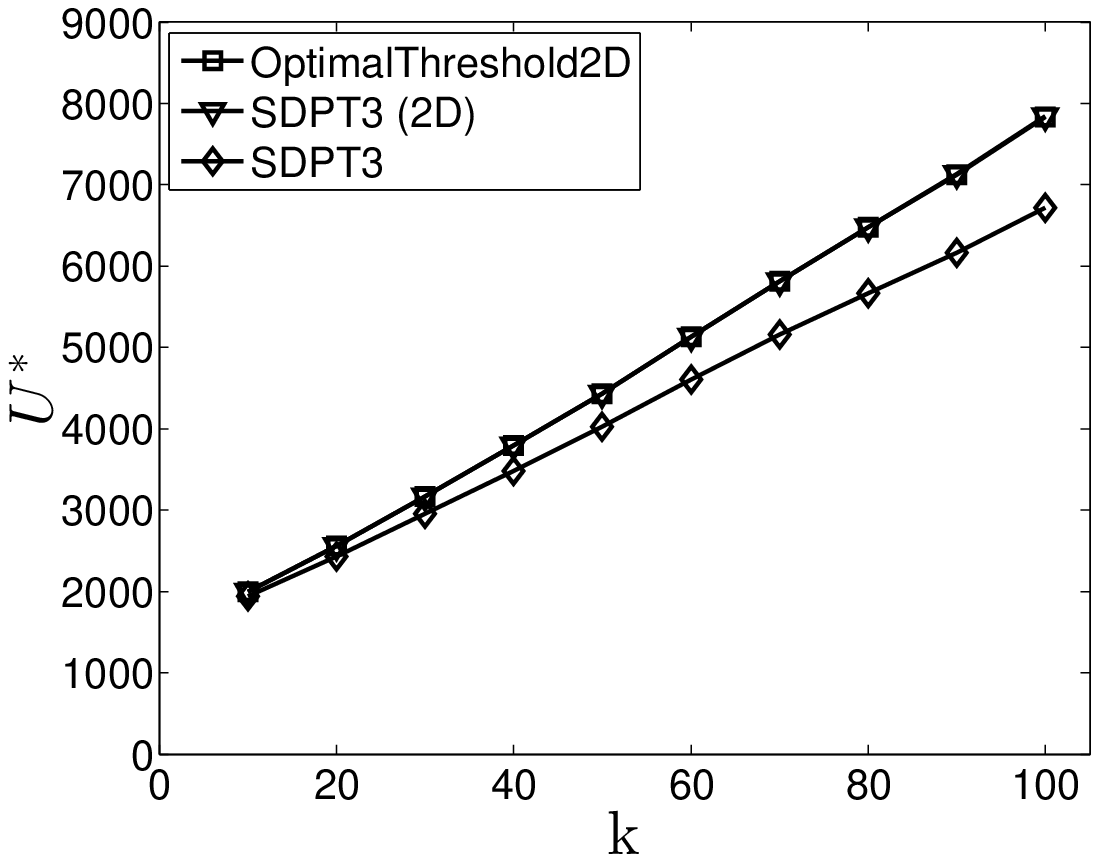}
\caption{Details of the costs in the case of uniform distribution of infection rates (see \rev{Tab.~\ref{table:perfhet})}.}\label{fig:perf}
\end{figure}

\fdpr{In the final set of experiments we study the case of a complete bipartite graph.} \rev{We consider that the community whose curing rate is $\delta_0$, to which we refer to as the  central community, has a fixed dimension $k_0=50$; instead, for the so-called terminal community, with $\delta_1$ curing rate, we consider increasing size $k_1=1, 50, 100, 150, 200$}. \rev{In Fig.~\ref{ratioBip}a) we report on the
  ratio between the cost obtained by using the uniform curing policy, namely $U_u$, and the optimal cost $U^*$ obtained by means of the~\texttt{OptimalThreshold2D}, in the case of equal coefficients $c_0=c_1=1$}.
	\rev{As expected, we can see that when $k_1=1$ we obtain an advantage in the use of the two-level
  curing strategy. Clearly, when the two communities have the same size there is no difference between the two costs; the ratio starts to grow again as the asymmetry in terms of communities dimensions, starts to increase again.} 
	
	\rev{In Fig.~\ref{ratioBip}b), we compare the effect of having two different cost coefficients, precisely $c_1=4 c_0$, with the case where they are the same, namely $c_0=c_1=1$, considering that their total amount is given and fixed a priori, that is $c_0+c_1=2$.
	A bias in the
  cost coefficient towards terminal community has a beneficial consequence because the optimal cost is lower than the case of equal coefficients, the advantage increases the larger the size of the terminal community. However, it is interesting to note that we would have obtained the same optimal cost if we interchanged the two coefficients, namely $c_0=4c_1$ (see example \ref{subset:example}).
	Basically, due to the specific topology of the network, we have an advantage in considering different cost coefficients for the two communities, instead of having them equal, if their total sum is fixed. }
	
\begin{figure*}[t]	
\centering 
\begin{minipage}{.47\textwidth}
\centering
\includegraphics[width=\textwidth]{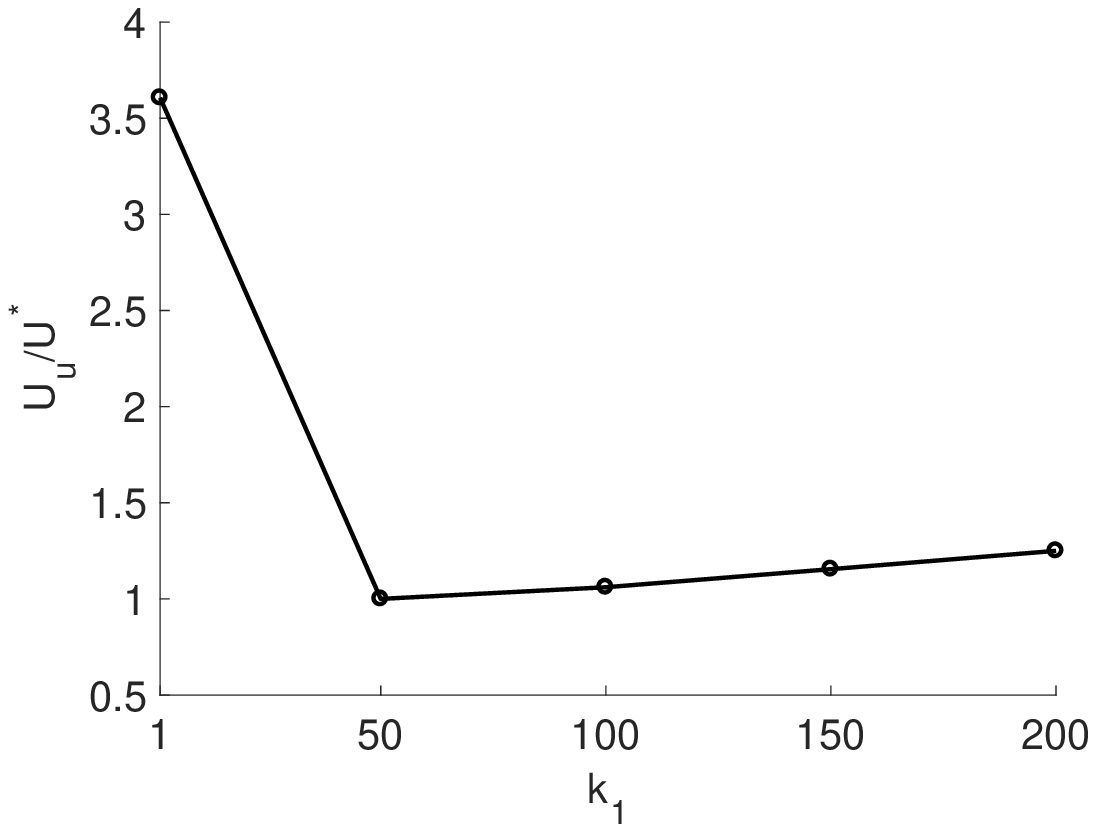}
\put(-192,143){a)}
\end{minipage}
\begin{minipage}{.47\textwidth}
\centering
\includegraphics[width=\textwidth]{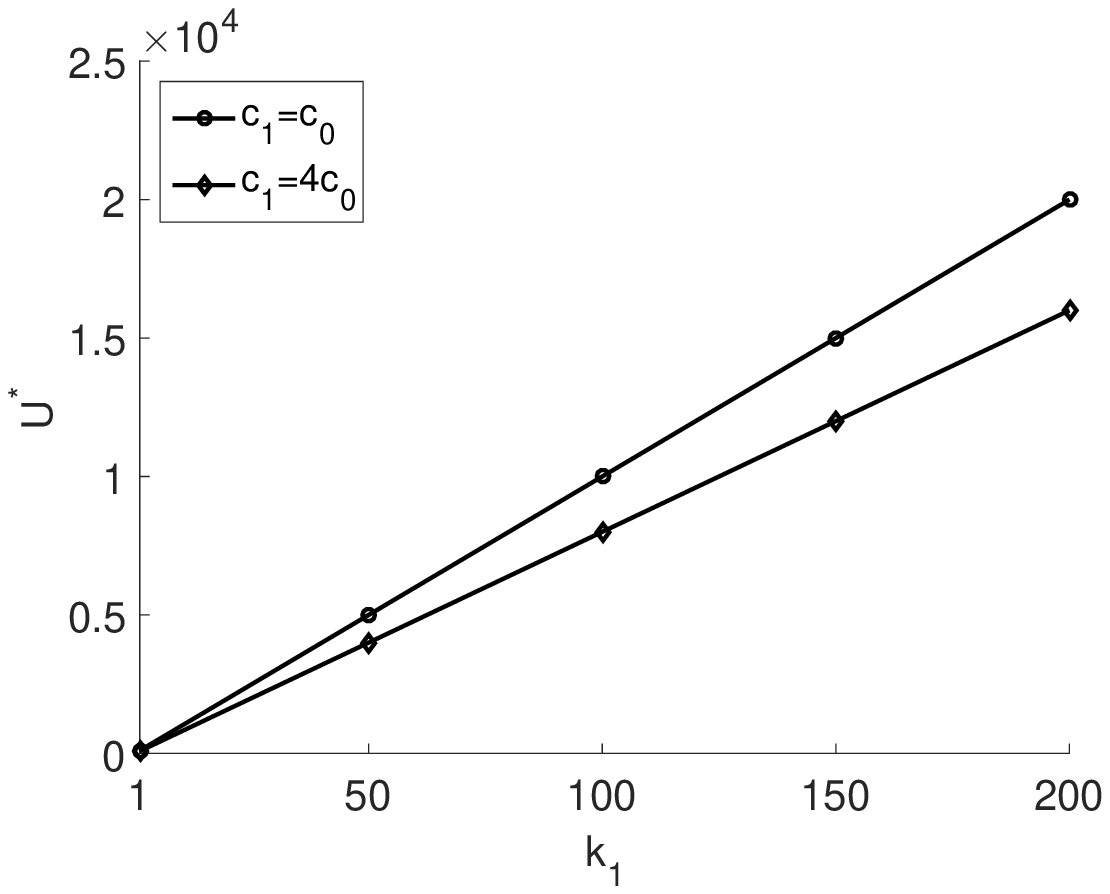}
\put(-194,143){b)}
\end{minipage}%
\caption{(a) \rev{Ratio $U_u/U^*$ in the case of complete bipartite graphs 
 for increasing size $k_1=1, 50, 100, 150, 200$ of the terminal community, and fixed size $k_0=50$ of the central community: $\beta=1$, $c_0=c_1=1$.
  (b) Cost function $U^*$ for increasing dimension $k_1$ of the terminal community and fixed size $k_0=50$, $\beta=1$. The two curves refer to the  cases $c_0=c_1=1$ and $c_1=4c_0$, respectively, in such a way that $c_0+c_1=2$.}}\label{ratioBip}
\end{figure*}


\section{Conclusions}\label{sec:concl}


\stef{We have studied the problem of finding a non-uniform allocation of curing resources, within a networked population, at the minimum cost possible to avoid the epidemic from persisting indefinitely in the network.
We have considered a mean-field approximation of an SIS model to study the diffusion of epidemics over a directed weighted graphs, capturing the possible asymmetric interactions, and the heterogeneity in the contagiousness.
We have reported on the necessary and sufficient
conditions for the extinction of the epidemics. These conditions are related to the sign of the stability modulus
of a matrix encoding for the network structure
and for the parameters of the model. Thus, such stability modulus represents the epidemic threshold of our system.
Consequently, we have formulated a convex optimization problem for determining a cost-optimal curing policy, via a semidefinite programming approach, involving, the spectral properties of the network. }

\stef{Then, we have specialized the theory to the case of equitable partitions,
in order to model heterogeneous community networks that possess a certain degree of regularity in their connectivity; this choice has been motivated since communities are
relevant non-trivial topological feature of complex networks, that often have a certain regularity in their structure. 
Thus, in this case, we were able to reduce the dimensionality of our optimization problem, 
that is useful since the size of many real-networks poses limitations in investigating their spectral properties. }

\stef{At last, we have discussed on the special case of a two-dimensional curing policy, that can
reflect, e.g., the case of different policy decisions for two different kinds of individual units
(male and female, younger and elders, small villages and cities, firewall/gateways or clients in an
enterprise network, etc.,...). With respect to this problem we have proposed
 an $\epsilon$-approximation algorithm with polynomial complexity in the input size. 
} 


\textbf{Fundings.} This work has been partially supported by the European Commission within the framework of the CONGAS project FP7-ICT-2011-8-317672 (see \url{http://www.congas-project.eu}

\section{Appendix}\label{app}
\subsection{Dimensionality reduction of the dynamical system \eqref{het}}\label{app:dim}

Let us define $q^T_{jm}$ as the element of $Q^T$ in position $(i,j)$. We know that %
\[
q^T_{jm}= \frac{k_j c^{out}_{jm}}{\sqrt{k_j k_m}},
\] hence 
\[
c^{out}_{jm}= \frac{\sqrt{k_j k_m}}{k_j} q^T_{jm}=\left(\frac{k_m}{k_j}\right)^{1/2} q^T_{jm}.
\]
Thus we can write \eqref{eqred} in the following matrix form
\begin{equation}\label{hetQ12}
\frac{d \overline{P}(t)}{dt}=\left( \tilde{Q}-\overline{D}\right)\overline{P}(t)-\diag(\overline{P}(t))\tilde{Q}\overline{P}(t),
\end{equation}
where  $\widetilde Q= \operatorname{diag}\left(\frac{1}{\sqrt{k_j}}\right) Q^T\operatorname{diag} (\sqrt{k_j})$.  
It is immediate that $\sigma(Q^T)= \sigma(\tilde{Q})$. 

By \cite[Corollary 4.2]{QEP}, irrespective of the initial conditions of nodes in the same community, 
it is sufficient to compute the positive steady-state vector $\overline{P}_{\infty}$ of the reduced system \eqref{hetQ12} to obtain that 
of the original system \eqref{het}. Indeed, as time elapses all nodes in the same community tend to have the same infection probabilities, thus the components of the steady-state vector $P_{\infty}$ corresponding to nodes in the same 
 community are equal.

\subsection{Proof of Lemma \ref{lemma1}}
\begin{proof}
\noindent i) We first prove that \stef{$A^TS^T=S^TQ^T$}. In fact, if $i \in V_h$, 
it holds
\begin{eqnarray}
(A^TS^T)_{i,j}&=& \frac{c^{out}_{hj}}{\sqrt{k_j}}, \\
(S^TQ^T)_{i,j}&=& \frac{1}{\sqrt{k_h}}q^T_{hj}= \frac{c^{out}_{hj}}{\sqrt{k_j}}. 
\end{eqnarray}

We further note that $(DS^T)_{ih}= (S^T\overline{D})_{ih}= \frac{1}{\sqrt{k_h}}\delta_h$, if $i \in V_h$, otherwise $(DS^T)_{ih}=0$. 
Thus the statement holds.

\noindent ii) By using the result in $i)$, the proof in \cite[Thm. 2.2]{godsil} applies.
\end{proof}

\subsection{Proof of proposition \ref{propthr}}

We first need some technical facts to prove Prop. \ref{propthr}.

\begin{proposition}\label{proprieties1}
Let $A$ be an $n \times n$ irreducible and non-negative matrix  and let $D=\diag(\delta_1, ...,\delta_n)$. Then it holds:\\
\noindent i. $A-D$ is irreducible, for each $(\delta_1, ...,\delta_n)$.\\
\noindent ii. There exists an eigenvector $w$ of $A-D$ such that $w > 0$ (i.e. each component $w_i>0$, $i=1, \dots, n$) and the corresponding eigenvalue is $r(A-D)$, for each $(\delta_1, ...,\delta_n)$.

\end{proposition}
\begin{proof}
\noindent i. From~\cite{Stab}: a $n \times n$ matrix $A$ is said to be irreducible if for any proper subset $S \subseteq \left\{1, \ldots, n\right\}$ there exists $i \in S$ and $j \in S'=\left\{1, \ldots, n\right\}-S$ such that $a_{ij} \neq 0$; since $A$ is irreducible, the definition applies immediately to $A-D$;\\
\noindent ii. See  \cite[Lemma~4.2]{Stab}.
\end{proof}

With these background statements we prove Prop. \ref{propthr}.
\begin{proof}
Basically, by Theorem \ref{threshhet}, we have to show that
\begin{equation}\label{eigA}
 r(A^T - D)=r(\tilde{Q}-\overline D)=r(Q^T- \overline D).
\end{equation}
We first prove that
\begin{equation}\label{eqeigA}
r( Q^T-\overline D)=r( A^T - D).
\end{equation}

Let $c \in \mathbb{R}$ such that both $a^T_{zz}- \delta_z+ c \geq 0$, for all $z=1, \ldots, N$ and $q^T_{ii} -\overline \delta_i +c\geq 0$  for all $i=1 \ldots, n$. 
Let us define $A^T-D^T+cI_{N \times N}=\hat{A^T}$ and $Q^T- \overline{D} + cI_{n \times n}=\hat{ Q^T}$.
 $\hat{A^T}$ and $\hat{Q^T}$ are non negative and irreducible matrices (see $i)$ in Proposition \ref{proprieties1}).
We order the eigenvalues of $\hat{ Q^T}$ so that $|\lambda_1(\hat{ Q^T})|\geq |\lambda_2(\hat{ Q^T})| \geq \ldots \geq|\lambda_n(\hat{ Q^T})|$, and similarly for $\hat{ A^T}$.
By the Perron-Frobenius theorem \stef{\cite[Chapter 8]{horn}}, the eigenvalue of maximum modulus of an irreducible and non negative matrix is
 real and positive and its corresponding eigenvector, the Perron vector, is
 the unique (up to a factor) strictly positive eigenvector of the matrix.
Hence there exists an eigenvector $w > 0$ of $\hat{ Q^T}$ corresponding to $\lambda_1(\hat{ Q^T})$, i.e. $w_i>0$, for all $i=1, ..., n$.
Now, by $ii)$ in Lemma \ref{lemma1} and since $S^T I_{n \times n}=I_{N \times N} S^T$, we have that
$S^T w>0$ is the eigenvector of $\hat{ A^T}$ corresponding to $\lambda_1(\hat{ Q^T})$. However, because $S^T w$ is strictly positive, it must be the Perron vector of $\hat{ A^T}$, consequently $\lambda_1(\hat{ A^T})=\lambda_1(\hat{ Q^T})$.
Since $$r(\hat{ Q^T})= \lambda_1(\hat{ Q^T})= \lambda_1(\hat{ A^T})=r(\hat{ A^T}),$$ and 
 \begin{equation}\label{real}
r(Q^T-\overline D)+c=r(\hat{ Q^T})=r(\hat{ A^T})=r(A^T-D) + c,
\end{equation}
 \eqref{eqeigA} is proved.  
Now we prove that 
\begin{equation}\label{eqeigQ}
r(\tilde{Q}-\overline D)=r(Q^T-\overline D).
\end{equation}
Let the matrix $\Lambda=\diag(k_i)$, $i=1, \ldots, n$. For any $n$-dimensional vector $v$ and scalar $\lambda \in \mathbb{C}$ we have that
$$ \left(\tilde{Q} - \overline D\right)v = \lambda v \Leftrightarrow \left(  \Lambda^{-\frac12}Q \Lambda^{\frac12}-\overline D\right)v=\lambda v \Leftrightarrow $$
$$\left(Q \Lambda^{\frac12} - \overline D \Lambda^{\frac12} \right)v = \lambda \Lambda^{\frac12}v\Leftrightarrow $$
$$\left(Q -\overline D\right) \left(\Lambda^{\frac12} v\right)
=\lambda \left(\Lambda^{\frac12}v\right),$$
hence $\lambda \in \sigma(\tilde{Q} - \overline D) \Longleftrightarrow  \lambda \in \sigma( Q - \overline D)$, so that \eqref{eqeigQ} is verified.
In conclusion, from \eqref{eqeigQ} and \eqref{eqeigA} it follows \eqref{eigA}.
\end{proof}

\subsection{Proof of proposition \ref{propr2}}

\begin{proof}
\fdpr{\noindent i. Let $\delta_i=0$ for some $i=1,\ldots n$ and assume that $\lambda_1(A-D) < 0$: for the vector ${\mathbf e}_i$ of the canonical basis it holds ${\mathbf e}_i^T \rev{(A - D)}{\mathbf e}_i={\mathbf e}_i^T A {\mathbf e}_i  \geq 0$ which is a contradiction.}\\
\fdpr{\noindent ii. The eigenvalues of \rev{such kind of} matrix vary with continuity with the entries of the matrix \cite[Appendix D]{horn}.}\\
\fdpr{iii. Let us consider $c>0$ such that $-d_{ii}+c\geq 0$ for all $i=1,\ldots,n$. \rev{Then, $A-D+cI$ is non-negative irreducible and $\lambda_1(A-D+c\cdot I_n)$ is actually its Perron-Frobenius eigenvalue \cite[Chapter 8]{horn}}.
Now, we can write for any $\epsilon>0$
\[
\lambda_1(A-D+\epsilon\, \rev{\diag({\mathbf e_i})})=\lambda_1(A-D+\epsilon \,  \rev{\diag({\mathbf e_i})}+c\, I_n)-c > \lambda_1(A-D +c\,I_n)-c=\lambda_1(A-D),
\]}
\rev{where the strict majorization holds because the Perron-Frobenius eigenvalue is strictly increasing in any entry of the matrix \cite{varga2009matrix,horn}}. 

\end{proof}

\subsection{Examples} \label{subset:example}

\noindent 1. A simple example of the optimal solution for the case of a community network 
is that of a star graph, where we have two communities, one formed by the central node and the 
other by the leaf nodes. 
Assuming that the infection rate is $\beta$, we have to find the value of 
$\delta_0$ and $\delta_1$ for which $ \beta Q- D$ has the maximal eigenvalue which is equal 
to zero. The characteristic polynomial of $\beta Q- D$ for a star graph with $k$ leaves is
\[
p_{\lambda}(\beta Q- D)=\lambda^2 + (\delta_0 + \delta_1)\lambda +\delta_0\delta_1 - \beta^2 k.
\]
We observe that $\lambda=0$ belongs to the spectrum of $\beta Q- D$ if and only if $\delta_0=\beta^2  k/\delta_1$.
 This also ensures that the second eigenvalue is negative and, consequently, $\lambda=0$ must be the largest eigenvalue  
of $\beta Q- D$. Thus replacing the value of $\delta_0$ obtained above in

$$U(\delta_0,\delta_1)= c_0 \delta_0 + c_1 k \delta_1,$$

and setting to zero the following derivative

$$U'(\delta_1)=-\frac{c_0  \beta^2 k^3}{\delta_1^2}+ c_1 k,$$

we have that the linear cost optimization is solved for
\begin{equation}\label{eq:immtwolev}
\delta_0 = \beta k \sqrt{\frac{c_1}{c_0}},  \quad    \delta_1= \beta \sqrt{\frac{c_0 }{c_1}},
\end{equation}
which in turn provides the optimal cost 
\begin{equation*}\label{eq:starcost}
U^*=\beta k \Big ( c_0\sqrt{\frac{c_1}{c_0}} + c_1\sqrt{\frac{c_0}{c_1}}\Big )= 2 \beta k \left(\sqrt{c_1 c_0}\right). 
\end{equation*}
We observe that the optimal cost is linear in the terminal community size $k$.

In the case of a uniform curing policy, where all nodes are cured at rate $\delta$, we have that 
the value of $\delta$ such that $\lambda=0$ is the largest eigenvalue of $\beta Q - D$ is equal to $\beta \sqrt{k}$, thus the 
value $U_u$ of the total cost is 
\[
U_u = \beta \sqrt{k} (c_0 + c_1 k). 
\]

It is easy to see that the ratio $U_u/U^*$ is increasing in $(1, \infty)$, moreover we can observe that
\begin{equation}\label{eq:ratio}
\frac{U_u}{U^*}  =  O(\sqrt k).
\end{equation}
Hence it is clear that we have an advantage in considering a two-level curing policy, with respect to the uniform 
curing policy.  

\noindent 2. Now we consider an interconnected star network with two linked central nodes, where each terminal community has the same number of elements $k$. We set $\beta$ as the infection rate between the central nodes and $\eps \beta$ the infection rate between a central node and a node in its adjacent terminal community, where $\eps > 0$.
After computing the characteristic polynomial of 
$\beta Q- D$ we can see that the zero eigenvalue belongs to the spectrum of $\beta Q- D$  provided that $$\delta^2_0 \delta^2_1 - 2 \beta^2 \delta_0 \delta_1 \eps^2 k + \eps^4 \beta^4 k^2- \beta^2 \delta_1^2 =0.$$ 
The values of $\delta_0$ for which $\lambda=0$ corresponds to the largest eigenvalue of $\beta Q- D$  is equal to $\delta_0 = \frac{\beta^2 \eps^2 k}{\delta_1} +\beta$ and the linear cost optimization is solved for
\[
\delta_0 =  \beta \left(\eps k \sqrt{\frac{c_1}{c_0}}+1\right), \qquad \delta_1=\eps \beta \sqrt{\frac{c_0}{c_1}}.
\]
Consequently the optimal cost is
\begin{equation}\label{twostarcost}
U^*= 2 \beta c_0 \left(\eps k \sqrt{\frac{c_1}{c_0}} + 1 \right) + 2 c_1 k \sqrt{\frac{c_0}{c_1}} \eps \beta = 2 \beta \sqrt{c_0c_1}(\eps (k+1) + c_0).
\end{equation}

In the case of a uniform curing policy we have that the value of $\delta$ such that $\lambda=0$ is the largest eigenvalue is $(\beta+\sqrt{\beta^2 + 4 \beta^2 \eps^2 k})/2$ and the value of the total cost is
\[
 U_u=\hspace{-0.2 em}c_0\hspace{-0.2 em}\left(\hspace{-0.2 em}\beta\hspace{-0.2 em} +\hspace{-0.2 em} \sqrt{\beta^2\hspace{-0.2 em} +\hspace{-0.2 em} 4\beta^2 \eps^2 k}\hspace{-0.2 em}\right)\hspace{-0.2 em} +\hspace{-0.2 em} %
c_1 k \left(\hspace{-0.2 em} \beta\hspace{-0.2 em} +\hspace{-0.2 em} \sqrt{\beta^2\hspace{-0.2 em} +\hspace{-0.2 em} 4\beta^2 \eps^2 k}\hspace{-0.2 em}\right).
\]
 
The ratio $U_u/U^*$ is increasing in $(0,\infty)$,
and again we have that
\[
\frac{U_u}{U^*} = O(\sqrt k).
\]

\noindent 3. Now we consider a complete bipartite graph. Basically, we have two communities and we denote by 
$k_0$ the number of elements in the community whose nodes have recovery rate $\delta_0$ and $k_1$ the 
number of elements in the community whose nodes has recovery rate $\delta_1$. The optimal 
curing rates are
\[
\delta_0= \beta k_1 \sqrt{\frac{c_1}{c_0}},  \qquad \delta_1= \beta k_0 \sqrt{\frac{c_0}{c_1}},
\]
and the optimal cost is
\[
U^*= c_0 k_0 \beta k_1 \sqrt{\frac{c_1}{c_0}} + c_1 k_1 \beta k_0 \sqrt{\frac{c_0}{c_1}}.
\]
In the case of a uniform curing policy the value of $\delta$ such that $\lambda=0$ is the largest eigenvalue is
\[
\delta= \beta \sqrt{k_0 k_1},
\]
and the cost is
\[
U_u= c_0 k_0 \beta \sqrt{k_0 k_1}+ c_1 k_1 \beta \sqrt{k_0 k_1}.
\]

In this case the asymptotic behavior of $U_U/U^*$ for high values of $k_0$ and $k_1$ 
depends on the direction in which we move, thus we can not say anything in this regard.

\flushend
%
\bibliographystyle{unsrt}

\begin{thebibliography}{10}

\bibitem{PietSurvey}
Romualdo Pastor-Satorras, Claudio Castellano, Piet Van~Mieghem, and Alessandro
  Vespignani.
\newblock Epidemic processes in complex networks.
\newblock {\em Reviews of modern physics}, 87(3):925, 2015.

\bibitem{wan2008designing}
Yan Wan, Sandip Roy, and Ali Saberi.
\newblock Designing spatially heterogeneous strategies for control of virus
  spread.
\newblock {\em Systems Biology, IET}, 2(4):184--201, 2008.

\bibitem{prakash2013fractional}
B~Aditya Prakash, Lada Adamic, TJ~Iwashyna, Hanghang Tong, and Christos
  Faloutsos.
\newblock Fractional immunization in networks.
\newblock {\em Austin, Texas, USA}, 2013.

\bibitem{borgs2010distribute}
Christian Borgs, Jennifer Chayes, Ayalvadi Ganesh, and Amin Saberi.
\newblock How to distribute antidote to control epidemics.
\newblock {\em Random Structures \& Algorithms}, 37(2):204--222, 2010.

\bibitem{gourdin2011optimization}
Eric Gourdin, Jasmina Omic, and Piet Van~Mieghem.
\newblock Optimization of network protection against virus spread.
\newblock In {\em Design of Reliable Communication Networks (DRCN), 2011 8th
  International Workshop on the}, pages 86--93. IEEE, 2011.

\bibitem{sahneh2012optimal}
Faryad~Darabi Sahneh and Caterina~M Scoglio.
\newblock Optimal information dissemination in epidemic networks.
\newblock In {\em Decision and Control (CDC), 2012 IEEE 51st Annual Conference
  on}, pages 1657--1662. IEEE, 2012.

\bibitem{VanMieghem2009}
P.~Van~Mieghem, J.~Omic, and R.~Kooij.
\newblock Virus spread in networks.
\newblock {\em Networking, IEEE/ACM Tran. on}, 17(1):1--14, Feb 2009.

\bibitem{Pappas}
V.~M. Preciado, M.~Zargham, C.~Enyioha, A.~Jadbabaie, and G.~Pappas.
\newblock Optimal vaccine allocation to control epidemic outbreaks in arbitrary
  networks.
\newblock {\em CoRR}, abs/1303.3984, 2013.

\bibitem{preciado2014optimal}
Victor~M Preciado, Michael Zargham, Chinwendu Enyioha, Ali Jadbabaie, and
  George~J Pappas.
\newblock Optimal resource allocation for network protection against spreading
  processes.
\newblock {\em Control of Network Systems, IEEE Transactions on}, 1(1):99--108,
  2014.

\bibitem{enyioha2015distributed}
Chinwendu Enyioha, Ali Jadbabaie, Victor Preciado, and George~J Pappas.
\newblock Distributed resource allocation for epidemic control.
\newblock {\em arXiv preprint arXiv:1501.01701}, 2015.

\bibitem{drakopoulos2015network}
Kimon Drakopoulos, Asuman Ozdaglar, and John~N Tsitsiklis.
\newblock When is a network epidemic hard to eliminate?
\newblock {\em arXiv preprint arXiv:1510.06054}, 2015.

\bibitem{drakopoulos2015lower}
Kimon Drakopoulos, Asuman Ozdaglar, and John~N Tsitsiklis.
\newblock A lower bound on the performance of dynamic curing policies for
  epidemics on graphs.
\newblock {\em arXiv preprint arXiv:1510.06055}, 2015.

\bibitem{boccaletti2006}
Stefano Boccaletti, Vito Latora, Yamir Moreno, Martin Chavez, and D-U Hwang.
\newblock Complex networks: Structure and dynamics.
\newblock {\em Physics reports}, 424(4):175--308, 2006.

\bibitem{Schwenk}
Allen~J Schwenk.
\newblock Computing the characteristic polynomial of a graph.
\newblock In {\em Graphs and Combinatorics}, pages 153--172. Springer, 1974.

\bibitem{godsil}
C.~Godsil.
\newblock Feasibility conditions for the existence of walk-regular graphs.
\newblock {\em Linear Algebra and its Applications}, 30:15--61, 1980.

\bibitem{EvolDelio}
Delio Mugnolo.
\newblock {\em Semigroup methods for evolution equations on networks}.
\newblock Springer, 2014.

\bibitem{watts2005multiscale}
Duncan~J Watts, Roby Muhamad, Daniel~C Medina, and Peter~S Dodds.
\newblock Multiscale, resurgent epidemics in a hierarchical metapopulation
  model.
\newblock {\em Proceedings of the National Academy of Sciences of the United
  States of America}, 102(32):11157--11162, 2005.

\bibitem{may1984spatial}
Robert~M May and Roy~M Anderson.
\newblock Spatial heterogeneity and the design of immunization programs.
\newblock {\em Mathematical Biosciences}, 72(1):83--111, 1984.

\bibitem{riley2003transmission}
Steven Riley, Christophe Fraser, Christl~A Donnelly, Azra~C Ghani, Laith~J
  Abu-Raddad, Anthony~J Hedley, Gabriel~M Leung, Lai-Ming Ho, Tai-Hing Lam,
  Thuan~Q Thach, et~al.
\newblock Transmission dynamics of the etiological agent of sars in hong kong:
  impact of public health interventions.
\newblock {\em Science}, 300(5627):1961--1966, 2003.

\bibitem{QEP}
Stefano Bonaccorsi, Stefania Ottaviano, Delio Mugnolo, and Francesco {De
  Pellegrini}.
\newblock Epidemic outbreaks in networks with equitable or almost-equitable
  partitions.
\newblock {\em SIAM Journal of Applied Mathematics}, 75(6):2421 -- 2443, 2015.

\bibitem{scoglio}
Faryad~Darabi Sahneh, Caterina Scoglio, and Piet Van~Mieghem.
\newblock Generalized epidemic mean-field model for spreading processes over
  multilayer complex networks.
\newblock {\em Networking, IEEE/ACM Transactions on}, 21(5):1609--1620, 2013.

\bibitem{VanMieghem2013}
Piet Van~Mieghem.
\newblock Decay towards the overall-healthy state in sis epidemics on networks.
\newblock {\em arXiv preprint arXiv:1310.3980}, 2013.

\bibitem{Draief2010}
Moez Draief and Laurent Massouli.
\newblock {\em Epidemics and rumours in complex networks}.
\newblock Cambridge University Press, 2010.

\bibitem{van2016approximate}
P~Van~Mieghem.
\newblock Approximate formula and bounds for the time-varying
  susceptible-infected-susceptible prevalence in networks.
\newblock {\em Physical Review E}, 93(5):052312, 2016.

\bibitem{VanMieghem2014}
Piet Van~Mieghem.
\newblock Exact markovian sir and sis epidemics on networks and an upper bound
  for the epidemic threshold.
\newblock {\em arXiv preprint arXiv:1402.1731}, 2014.

\bibitem{VanMieghem2012a}
Piet Van~Mieghem.
\newblock Exact markovian sir and sis epidemics on networks and an upper bound
  for the epidemic threshold.
\newblock {\em arXiv preprint arXiv:1402.1731}, 2014.

\bibitem{CatorPositivecorrelations}
E.~Cator and P.~Van~Mieghem.
\newblock Nodal infection in {M}arkovian {SIS} and {SIR} epidemics on networks
  are non-negatively correlated.
\newblock {\em Physical Review E}, 89(5):052802, 2014.

\bibitem{Inhom}
Piet Van~Mieghem and Jasmina Omic.
\newblock In-homogeneous virus spread in networks.
\newblock {\em arxiv: 1306.2588}, 2013.

\bibitem{Stab}
Ana Lajmanovich and James~A. Yorke.
\newblock A deterministic model for {G}onorrhea in a non-homogeneous
  population.
\newblock {\em Mathematical Biosciences}, 28(34):221 -- 236, 1976.

\bibitem{hupert2004community}
Nathaniel Hupert, Jason Cuomo, Mark~A Callahan, Alvin~I Mushlin, and Stephen~S
  Morse.
\newblock {\em Community-based mass prophylaxis: A planning guide for public
  health preparedness}.
\newblock US Department of Health and Human Services, Agency for Healthcare
  Research and Quality, 2004.

\bibitem{boyd96}
S.~P. Boyd.
\newblock Semidefinite programming.
\newblock {\em SIAM review}, 38:49--95, 1994.

\bibitem{boyd}
S.P. Boyd and L.~Vandenberghe.
\newblock {\em Convex optimization}.
\newblock Cambridge University Press, 2004.

\bibitem{berman1994nonnegative}
Abraham Berman and Robert~J Plemmons.
\newblock {\em Nonnegative matrices in the mathematical sciences}.
\newblock SIAM, 1994.

\bibitem{zhang2011matrix}
Fuzhen Zhang.
\newblock {\em Matrix theory: basic results and techniques}.
\newblock Springer Science \& Business Media, 2011.

\bibitem{SDPT3}
R.~H. T{\"u}t{\"u}nc{\"u}, K.~C. Toh, and M.~J. Todd.
\newblock Solving semidefinite-quadratic-linear programs using {SDPT3}.
\newblock {\em Mathematical Programming}, 95(2):189--217, 2003.

\bibitem{pellis2015seven}
Frank Ball, Tom Britton, Thomas House, Valerie Isham, Denis Mollison, Lorenzo
  Pellis, and Gianpaolo~Scalia Tomba.
\newblock Seven challenges for metapopulation models of epidemics, including
  households models.
\newblock {\em Epidemics}, 10:63--67, 2015.

\bibitem{Hanski}
Ilkka Hanski and Otso Ovaskainen.
\newblock Metapopulation theory for fragmented landscapes.
\newblock {\em Theoretical population biology}, 64(1):119--127, 2003.

\bibitem{Masuda2010}
Naoki Masuda.
\newblock Effects of diffusion rates on epidemic spreads in metapopulation
  networks.
\newblock {\em New Journal of Physics}, 12(9):093009, 2010.

\bibitem{ball1997epidemics}
Frank Ball, Denis Mollison, and Gianpaolo Scalia-Tomba.
\newblock Epidemics with two levels of mixing.
\newblock {\em The Annals of Applied Probability}, pages 46--89, 1997.

\bibitem{ross2010calculation}
Joshua~V Ross, Thomas House, and Matt~J Keeling.
\newblock Calculation of disease dynamics in a population of households.
\newblock {\em PLoS One}, 5(3):e9666, 2010.

\bibitem{ball2002general}
Frank Ball and Peter Neal.
\newblock A general model for stochastic sir epidemics with two levels of
  mixing.
\newblock {\em Mathematical biosciences}, 180(1):73--102, 2002.

\bibitem{pellis2012reproduction}
Lorenzo Pellis, Frank Ball, and Pieter Trapman.
\newblock Reproduction numbers for epidemic models with households and other
  social structures. i. definition and calculation of r 0.
\newblock {\em Mathematical biosciences}, 235(1):85--97, 2012.

\bibitem{ball2008network}
Frank Ball and Peter Neal.
\newblock Network epidemic models with two levels of mixing.
\newblock {\em Mathematical biosciences}, 212(1):69--87, 2008.

\bibitem{frank2009}
Ball Frank, Sirl David, Trapman Pieter, et~al.
\newblock Threshold behaviour and final outcome of an epidemic on a random
  network with household structure.
\newblock {\em Advances in Applied Probability}, 41(3):765--796, 2009.

\bibitem{wang2013effect}
Huijuan Wang, Qian Li, Gregorio D?Agostino, Shlomo Havlin, H~Eugene Stanley,
  and Piet Van~Mieghem.
\newblock Effect of the interconnected network structure on the epidemic
  threshold.
\newblock {\em Physical Review E}, 88(2):022801, 2013.

\bibitem{Bonaccorsi}
Stefano Bonaccorsi, Stefania Ottaviano, Francesco {De Pellegrini}, Annalisa
  Socievole, and Piet Van~Mieghem.
\newblock Epidemic outbreaks in two-scale community networks.
\newblock {\em Phys. Rev. E}, 90:012810, Jul 2014.

\bibitem{girvan2002community}
Michelle Girvan and Mark~EJ Newman.
\newblock Community structure in social and biological networks.
\newblock {\em Proceedings of the national academy of sciences},
  99(12):7821--7826, 2002.

\bibitem{stewart2003symmetry}
Ian Stewart, Martin Golubitsky, and Marcus Pivato.
\newblock Symmetry groupoids and patterns of synchrony in coupled cell
  networks.
\newblock {\em SIAM Journal on Applied Dynamical Systems}, 2(4):609--646, 2003.

\bibitem{golubitsky2005patterns}
Martin Golubitsky, Ian Stewart, and Andrei T{\"o}r{\"o}k.
\newblock Patterns of synchrony in coupled cell networks with multiple arrows.
\newblock {\em SIAM Journal on Applied Dynamical Systems}, 4(1):78--100, 2005.

\bibitem{rahmani2009controllability}
Amirreza Rahmani, Meng Ji, Mehran Mesbahi, and Magnus Egerstedt.
\newblock Controllability of multi-agent systems from a graph-theoretic
  perspective.
\newblock {\em SIAM Journal on Control and Optimization}, 48(1):162--186, 2009.

\bibitem{aguilar2016almost}
Cesar~O Aguilar and Bahman Gharesifard.
\newblock On almost equitable partitions and network controllability.
\newblock In {\em American Control Conference (ACC), 2016}, pages 179--184.
  IEEE, 2016.

\bibitem{Nemirovski}
Arkadi Nemirovski.
\newblock Advances in convex optimization: Conic programming.
\newblock In {\em In Proceedings of International Congress of Mathematicians},
  pages 413--444, 2007.

\bibitem{barabasi1999}
Albert-L{\'a}szl{\'o} Barab{\'a}si and R{\'e}ka Albert.
\newblock Emergence of scaling in random networks.
\newblock {\em science}, 286(5439):509--512, 1999.

\bibitem{Li2012}
Cong Li, Ruud van~de Bovenkamp, and Piet Van~Mieghem.
\newblock Susceptible-infected-susceptible model: A comparison of n-intertwined
  and heterogeneous mean-field approximations.
\newblock {\em Physical Review E}, 86(2):026116, 2012.

\bibitem{Aho_DesiCompAlgos}
{A. V.} Aho, {J. E.} Hopcroft, and {J. D.} Ullman.
\newblock {\em The Design and Analysis of Computer Algorithms}.
\newblock Addison-Wesley, 1974.

\bibitem{horn}
Roger~A. Horn and Charles~R. Johnson.
\newblock {\em Matrix Analysis}.
\newblock Cambridge University Press, New York, NY, USA, 2012.

\bibitem{varga2009matrix}
Richard~S Varga.
\newblock {\em Matrix iterative analysis}, volume~27.
\newblock Springer Science \& Business Media, 2009.

\end{thebibliography}

\end{document}